\documentclass[10pt]{article}
\usepackage{amssymb,amsmath,amsthm,indentfirst}
\usepackage{multicol}
\usepackage{graphicx}
\DeclareMathOperator{\Rea}{Re}
\DeclareMathOperator{\Ima}{Im}
\renewcommand{\leq}{\leqslant}
\renewcommand{\geq}{\geqslant}

\newtheorem{proposition}{Proposition}
\newtheorem{lemma}{Lemma}

\begin{document}
\large
\begin{center}
{\LARGE{\bf\sf A splitting higher order scheme}}
\end{center}
\begin{center}
{\LARGE{\bf\sf with discrete transparent boundary conditions}}
\end{center}
\begin{center}
{\LARGE{\bf\sf for the Schr\"odinger equation in a semi-infinite parallelepiped}}
\vskip 0.5cm
{\large Bernard Ducomet
\footnote[1]{DPTA/Service de Physique Nucl\'eaire, CEA/DAM/DIF Ile de France, BP 12, F--91297, Arpajon, France.
E-mail: \it{bernard.ducomet@cea.fr}},
{\large Alexander Zlotnik}
\footnote[2]{Department of Higher Mathematics at Faculty of Economics,
National Research University Higher School of Economics,
Myasnitskaya 20, 101000 Moscow, Russia.}
\footnote[3]{Department of Mathematical Modelling,
National Research University Moscow Power Engineering Institute,
Krasnokazarmennaya 14, 111250 Moscow, Russia. E-mail: \it{azlotnik2008@gmail.com}}
{\large and Alla Romanova}
\footnote[4]
{Department of Higher Mathematics at Faculty of Economics,
National Research University Higher School of Economics,
Myasnitskaya 20, 101000 Moscow, Russia.
E-mail: \it{avromm1@gmail.com}}}
\end{center}
\vskip 0.5cm
\begin{abstract}
\noindent An initial-boundary value problem for the $n$-dimensional ($n\geq 2$) time-dependent Schr\"odinger equation in a semi-infinite (or infinite) parallelepiped is considered.
Starting from the Numerov-Crank-Nicolson finite-difference scheme, we first construct higher order scheme with splitting space averages having much better spectral properties for $n\geq 3$.
Next we apply the Strang-type splitting with respect to the potential and, third, construct discrete transparent boundary conditions (TBC).
For the resulting method, the uniqueness of solution and the unconditional uniform in time $L^2$-stability (in particular, $L^2$-conservativeness) are proved.
Owing to the splitting, an effective direct algorithm using FFT (in the coordinate directions perpendicular to the leading axis of the parallelepiped) is applicable for general potential.
Numerical results on the 2D tunnel effect for a P\"{o}schl-Teller-like potential-barrier and a rectangular potential-well are also included.
\end{abstract}
\section{Introduction}
\label{sect1}
\large
The time-dependent Schr\"odinger equation with several space variables is crucial in quantum mechanics and electronics, nuclear and atomic physics, wave physics, etc. Often it should be solved in unbounded space domains.
\par Several approaches were developed and investigated for solving problems of such kind, in particular, see \cite{AABES08,ABM04,AES03,DiM97,Sch02a,SZ04}.
One of them exploits the so-called discrete transparent boundary conditions (TBCs) at artificial boundaries \cite{AES03,EA01}.
Its advantages are the complete absence of spurious reflections in practice as well as the rigorous mathematical background and stability results in theory.
\par The discrete TBCs for the Crank-Nicolson finite-difference scheme, the higher order Numerov-Crank-Nicolson scheme
and a general family of schemes on an infinite or semi-infinite strip were constructed and studied respectively in \cite{AES03,DZ06,DZ07}, \cite{SA08} and \cite{ZZ11,IZ11}.
All these schemes are implicit, so to implement them, solving of specific complex systems of linear algebraic equations is required at each time level.
\par The splitting technique is widely used to simplify numerical solving of the time-dependent Schr\"o\-dinger and related equations, in particular, see \cite{BM00,CMR13,GX11,G11,L08,TY10}.
The known Strang-type splitting with respect to the potential has been recently applied to the Crank-Nicolson and the Numerov-Crank-Nicolson scheme with the discrete TBCs in 2D case in \cite{DZZ13,ZR13}.
\par Higher order methods are important due to their ability to reduce computational costs essentially, and the Numerov-Crank-Nicolson scheme can be written in $n$-dimensional case as well. But we show that, for $n\geq 3$, the Numerov space operators lose their important spectral properties existing for $n=2$ so that the scheme becomes impractical.
\par In this paper, in the spirit of \cite{ZZ11,IZ11}, we first split these operators (in space) and recover the properties without reducing the higher order, for any $n\geq 2$. We second apply the Strang-type splitting in potential in time also conserving the higher order; the resulting scheme can be called ``double-(space-time)-splitting". For this scheme on an infinite space mesh in the semi-infinite parallelepiped, we prove the unconditional uniform in time $L^2$-stability together with the mass conservation law using combination of techniques from \cite{DZZ09,DZZ13,SA08}.
\par The discrete TBCs allow to restrict rigorously solutions of the schemes on infinite space meshes to finite ones; they can be written in several forms. Our form of the discrete TBC and its derivation for the ``double-splitting" scheme follow \cite{DZ06,DZZ09,ZZ11,IZ11}. This form simplifies the whole study and is computationally stable.
Notice that the discrete TBC is non-local and involves the discrete convolution in time together with the discrete Fourier operators in space directions perpendicular to the leading axis of the parallelepiped.
Exploiting an approach from \cite{ZR13} and suitable results from \cite{DZZ09}, we derive the uniqueness of solution to the ``double-splitting" scheme with the discrete TBC
and then its uniform in time $L^2$-stability (from the former stability result for the infinite space mesh). In particular, it is $L^2$-conservative.
\par Owing to the Strang-type splitting, an effective direct algorithm is considered to implement the method (for general potential) similar to those constructed in \cite{DZZ13,ZR13}. It uses the fast Fourier transform (FFT) in the perpendicular directions and a collection of independent 1D discrete Schr\"odinger problems at each time level.
\par The corresponding 2D numerical results on the tunnel effect for a P\"{o}schl-Teller-like potential-barrier and a rectangular potential-well are included.
In the case of the potential-barrier, we compare the ``double-splitting" scheme with the Numerov-Crank-Nicolson-Strang scheme from \cite{ZR13} and find that their errors are very close.
In the case the rectangular well, we present the behavior of the solution.
In both cases we check that rather rough space meshes can be actually used (due to the higher order in space) despite the large space derivatives or non-smoothness of the solution.

\section{The Schr\"odinger equation in a semi-infinite parallelepiped
 and its approximations of higher order in space}
\label{sect4}
\setcounter{equation}{0}
\setcounter{proposition}{0}
\setcounter{theorem}{0}
\setcounter{lemma}{0}
\setcounter{corollary}{0}
\setcounter{remark}{0}
We consider the multidimensional time-dependent Schr\"odinger equation
\begin{equation}
 i\hbar\frac{\partial\psi}{\partial t}
 =-\frac{\hbar^{\,2}}{2m_0}\,\Delta \psi+V\psi\ \
 \text{for}\ \ x=(x_1,\dots,x_n)\in\Pi_{\infty},\ \ t>0,
\label{a1}
\end{equation}
where $\Delta$ is the $n$-dimensional Laplace operator for $n\geq 2$,
$\Pi_{\infty}:=(0,\infty)\times \Pi_{\widehat{1}}$ is a semi-infinite parallelepiped, with $\Pi_{\widehat{1}}:=(0,X_2)\times\dots\times (0,X_n)$.
Hereafter $i$ is the imaginary unit, $\hbar>0$ and $m_0>0$ are physical constants, $\psi=\psi(x,t)$ is the unknown complex-valued wave function and $V(x)$ is a given real potential.
We also set $c_{\hbar}:=\frac{\hbar^{\,2}}{2m_0}$ for convenience.
\par We impose the following boundary condition, condition at infinity and initial condition
\begin{gather}
 \psi(\cdot,t)|_{\partial\Pi_{\infty}}=0,\ \
 \|\psi(x_1,\cdot,t)\|_{L^2(\Pi_{\widehat{1}})}\to 0\ \
 \text{as}\ \ x_1\to\infty,
 \ \ \text{for any}\ \ t>0,
\label{a2}\\[1mm]
 \psi|_{t=0}=\psi^0(x)\ \ \text{in}\ \ \Pi_{\infty}.
 \label{a3}
\end{gather}
We also assume that $V(x)$ is constant and $\psi_0(x)$ vanishes when $x_1$ is sufficiently large:
\begin{equation}
 V(x)=V_{\infty},\ \
 \psi^0(x)=0\ \ \text{for}\ \ x\in [X_0,\infty)\times\Pi_{\widehat{1}},
\label{a4}
\end{equation}
 for some $X_0>0$.
\par We introduce a uniform mesh $\overline{\omega}_{h,\infty}$ on $\overline{\Pi}_{\infty}$ with nodes
$x_{{\bf j}}=(j_1h_1,\dots,j_nh_n)$, $j_1\geq 0$,
$0\leq j_2\leq J_2,\dots,\ 0\leq j_n\leq J_n$
and steps $h_1=\frac{X_1}{J_1},\dots, h_n=\frac{X_n}{J_n}$, where $X_1>X_0$ and $2h_1\leq X_1-X_0$.
Let $\omega_{h,\infty}$ be its internal part consisting in the nodes
$x_{{\bf j}}$, $j_1\geq 1$, $1\leq j_2\leq J_2-1,\dots,1\leq j_n\leq J_n-1$ and let
$\Gamma_{h,\infty}:=\overline{\omega}_{h,\infty}\backslash\omega_{h,\infty}$ be its boundary. Hereafter $h=(h_1,\dots, h_n)$, $|h|$ is the length of $h$ and ${\bf j}=(j_1,\dots, j_n)$.
\par In the direction $x_k$, we exploit the backward, forward and central difference quotients
\[
 \bar{\partial}_kW_j:= \frac{W_j-W_{j-1}}{h_k},\ \
 \partial_kW_j:= \frac{W_{j+1}-W_{j}}{h_k},\ \
 \overset{\circ}{\partial}_k W_j:= \frac{W_{j+1}-W_{j-1}}{2h_k}
\]
as well as the Numerov average in $x_k$
\[
s_{Nk}W_j:= \frac{1}{12}\, W_{j-1}+\frac{5}{6}\, W_{j}+\frac{1}{12}\, W_{j+1}
=\Bigl(I+\frac{h_k^2}{12}\, \partial_k\bar{\partial}_k\Bigr)W_j,
\]
where $I$ is the unit operator.
\par We introduce also a non-uniform mesh $\overline{\omega}^{\,\tau}$ in time on $[0,\infty)$ with nodes
$0=t_0< t_1<\dots< t_m<\dots$, where $t_m\to\infty$ as $m\to\infty$, and steps $\tau_m=t_m-t_{m-1}$.
Let $\omega^{\tau}:=\overline{\omega}^{\,\tau}\backslash \{0\}$ and $\tau_{\max}=\sup_{m\geq 1} \tau_m$.
We exploit the backward difference quotient, the symmetric average and the backward shift in time
\[
 \bar{\partial}_t Y^m= \frac{Y^m-Y^{m-1}}{\tau_m},\ \
 \overline{s}_t Y^m= \frac{Y^{m-1}+Y^m}{2},\ \
 {\check Y}^m=Y^{m-1}.
\]
\par The simplest approximations of the Laplace operator in dimensions $n$ and $n-1$ (excluding $x_k$) are
\[
\Delta_h=\partial_1\bar{\partial}_1+\dots+\partial_n\bar{\partial}_n,
\ \ \Delta_{h,\widehat{k}}=\sum_{1\leq\ell\leq n,\, \ell\neq k}\partial_{\ell}\bar{\partial}_{\ell}.
\]
The Numerov-type approximation of the Laplace operator and the $n$-dimensional average are given by
\[
\Delta_{hN}=\Delta_h+\sum_{k=1}^n\frac{h_k^2}{12}\,\Delta_{h,\widehat{k}}\, \partial_k\bar{\partial}_k,
\ \ s_N=I+\frac{h_1^2}{12}\,\partial_1\bar{\partial}_1+
\dots+\frac{h_n^2}{12}\,\partial_n\bar{\partial}_n,
\]
and the following formula holds
\begin{equation}
\Delta_{hN}=\sum_{\ell=1}^n \Bigl(\partial_{\ell}\bar{\partial}_{\ell}
+\sum_{1\leq k\leq n,\, k\neq\ell}\frac{h_k^2}{12}\, \partial_{\ell}\bar{\partial}_{\ell}\partial_k\bar{\partial}_k
\Bigr)
=\sum_{\ell=1}^n s_{N\widehat{\ell}}\,\partial_{\ell}\bar{\partial}_{\ell},
\label{c1}
\end{equation}
with ${s_{N\widehat{k}}=I+\sum_{1\leq\ell\leq n,\,\ell\neq k}\frac{h_{\ell}^2}{12}\,\partial_{\ell}\bar{\partial}_{\ell}}$.
\par
We begin with an approximation for the Schr\"odinger equation (\ref{a1}) of the Numerov type in space and of the Crank-Nicolson type (i.e. two-level symmetric) in time given by
\begin{equation}
i\hbar s_N\bar{\partial}_t \Psi
=-c_{\hbar}\Delta_{hN}\overline{s}_t \Psi
+s_N\left( V\overline{s}_t \Psi\right)
\ \ \mbox{on}\ \ \omega_{h,\infty}\times\omega^{\tau}.
\label{c2}
\end{equation}
The corresponding approximation error
\[
i\hbar s_N\bar{\partial}_t \psi
+c_{\hbar}\Delta_{hN}\overline{s}_t \psi
-s_N\left( V\overline{s}_t \psi\right)
=O\left(\tau_{\max}^2+|h|^4\right)
\]
is of higher 4th order in $|h|$ and 2nd order in $\tau_{\max}$, for $\psi$ smooth enough.
\par Let us first check some important properties of $\Delta_{hN}$ and $s_N$.
For $1\leq p_1\leq J_1-1,\dots,1\leq p_n\leq J_n-1$, we set
\[
s^{({\bf p})}(x)=
\sin \frac{\pi p_1x_1}{X_1}\dots\sin \frac{\pi p_nx_n}{X_n},\ \ \text{with}\ \ {\bf p}:=(p_1,\dots,p_n).
\]
We have
\[
s_Ns^{({\bf p})}
=\lambda_{\bf p}[s_N]s^{({\bf p})},
\ \ -\Delta_{hN}s^{({\bf p})}
=\lambda_{\bf p}[-\Delta_{hN}]s^{({\bf p})},
\ \ -\Delta_{h}s^{({\bf p})}
=\lambda_{\bf p}[-\Delta_{h}]s^{({\bf p})}
\ \ \mbox{on}\ \omega_{h,\infty},
\]
with the eigenvalues
\begin{gather*}
\lambda_{\bf p}[s_N]
=1-\frac{1}{3}
\left(\sin^2 \frac{\pi p_1x_1}{2X_1}+\dots+\sin^2 \frac{\pi p_nx_n}{2X_n}\right),
\\[1mm]
\lambda_{\bf p}[-\Delta_{hN}]
=\lambda_{\bf p}[s_{N\widehat{1}}]\lambda_{p_1}^{(1)}
+\dots
+\lambda_{\bf p}[s_{N\widehat{n}}]\lambda_{p_n}^{(n)},\ \
\lambda_{\bf p}[-\Delta_{h}]
=\lambda_{p_1}^{(1)}
+\dots
+\lambda_{p_n}^{(n)},
\end{gather*}
where $\lambda_p^{(k)}=\Bigl( \frac{2}{h_k}\sin \frac{\pi p h_k}{2X_k}\Bigr)^2$ and $\lambda_{\bf p}[s_{N\widehat{k}}]$ is the $(n-1)$-dimensional version of $\lambda_{\bf p}[s_{N}]$ for $1\leq k\leq n$.
\par Clearly $\lambda_{\bf p}[s_{N}]\leq 1$ and $\lambda_{\bf p}[-\Delta_{hN}]\leq \lambda_{\bf p}[-\Delta_{h}]$ for any $n$.
\par If $n=2$, then
$\frac{1}{3}\leq \lambda_{\bf p}[s_{N}]$
and
$\frac{2}{3}\,\lambda_{\bf p}[-\Delta_{h}]\leq\lambda_{\bf p}[-\Delta_{hN}]$.
\par If $n=3$, then
$0< \lambda_{\bf p}[s_{N}]$
and
$\frac{1}{3}\,\lambda_{\bf p}[-\Delta_{h}]\leq\lambda_{\bf p}[-\Delta_{hN}]$ for any ${\bf p}$, but
$\lambda_{\bf p}[s_{N}]=O\left(|h|^2\right)$ for ${\bf p}=(J_1-1,\dots,J_n-1)$, i.e. $s_N$ is almost degenerate.

If $n\geq 4$, then even $\lambda_{\bf p}[s_{N}]<0$, in particular, for the same ${\bf p}=(J_1-1,\dots,J_n-1)$ and
sufficiently small $|h|$, and thus there exists no $c>0$ such that an inequality
$c\lambda_{\bf p}[-\Delta_{h}]\leq\lambda_{\bf p}[-\Delta_{hN}]$ holds uniformly in ${\bf p}$ and $h$.
Therefore the properties of $s_N$ for $n\geq 3$ and those of $\Delta_N$ for $n\geq 4$ are not natural, in contrast with the case $n=2$.
\par In order to construct operators with better properties, we suggest now to split the operators $s_{N\widehat{k}}$ in (\ref{c1}) and $s_N$ and to introduce
\[
\bar{\Delta}_{hN}
=\overline{s}_{N\widehat{1}}\,\partial_1\bar{\partial}_1
+\dots
+\overline{s}_{N\widehat{n}}\,\partial_n\bar{\partial}_n,
\ \ \mbox{with}\ \ \overline{s}_{N\widehat{k}}=\prod_{1\leq \ell\leq n,\, \ell\neq k} s_{N\ell},
\ \ \mbox{and}\ \ \overline{s}_N=s_{N1}\dots s_{Nn}.
\]
(Note that $\bar{\Delta}_{hN}=\Delta_{hN}$ for $n=2$.)
\par Then we pass from the discrete equation (\ref{c2}) to the following one
\begin{equation}
i\hbar \overline{s}_N\bar{\partial}_t \Psi
=-c_{\hbar}\bar{\Delta}_{hN}\overline{s}_t \Psi
+\overline{s}_N\left( V\overline{s}_t \Psi\right)
\ \ \mbox{on}\ \ \omega_{h,\infty}\times\omega^{\tau}
\label{f1}
\end{equation}
with the splitting average operators. Clearly $\overline{s}_N=s_N+O\left(|h|^4\right)$, and due to formula (\ref{c1}) one has also $\bar{\Delta}_{hN}=\Delta_{hN}+O\left(|h|^4\right)$.
Consequently the approximation error for the discrete Schr\"odinger equation (\ref{f1}) is of the same order $O\left(\tau_{\max}^2+|h|^4\right)$.
Notice that, for $n=2$ and constant $V$, discrete equation (\ref{f1}) is a particular case of one studied in \cite{ZZ11,IZ11}.
\par On the other hand, now one checks immediately that
\[
\overline{s}_N s^{({\bf p})}
=\lambda_{\bf p}[\overline{s}_N]s^{({\bf p})},
\ \ -\bar{\Delta}_{hN}s^{({\bf p})}
=\lambda_{\bf p}[-\bar{\Delta}_{hN}]s^{({\bf p})}\ \ \mbox{on}\ \ \omega_{h,\infty},
\]
with the eigenvalues satisfying
\[
\left(\frac{2}{3}\right)^{n}\leq \lambda_{\bf p}[\overline{s}_N]\leq 1,
\ \ \left(\frac{2}{3}\right)^{n-1}\lambda_{\bf p}[-\Delta_{h}]
\leq \lambda_{\bf p}[-\bar{\Delta}_{hN}]
\leq\lambda_{\bf p}[-\Delta_{h}]\ \ \text{for any}\ \ n.
\]
\par We supplement the discrete equation (\ref{f1}) with the boundary and initial conditions
\begin{equation}
\left. \Psi\right|_{\Gamma_{h,\infty}\times\omega^{\tau}}=0,
\ \ \Psi^0=\Psi^0_h\ \ \mbox{on}\ \ \overline{\omega}_{h,\infty}.
\label{f2}
\end{equation}
Hereafter the compatibility condition $\left. \Psi^0_h\right|_{\Gamma_{h,\infty}}=0$ is assumed.
\par
We further apply the known Strang-type splitting in the potential to the new scheme (\ref{f1}), (\ref{f2}) and get the following three-step scheme
\begin{gather}
 i\hbar\,\frac{\breve{\Psi}^m-\Psi^{m-1}}{\tau_m/2}
 =\Delta V\frac{\breve{\Psi}^m+\Psi^{m-1}}{2}\ \
 \text{on}\ \ \omega_{h,\infty},
\label{e1}\\[1mm]
 i\hbar\overline{s}_N\frac{\widetilde{\Psi}^m-\breve{\Psi}^m}{\tau_m}
 =-c_{\hbar}\bar{\Delta}_{hN}\frac{\widetilde{\Psi}^m+\breve{\Psi}^m}{2}
 +\overline{s}_N\Bigl(\widetilde{V}\frac{\widetilde{\Psi}^m+\breve{\Psi}^m}{2}\Bigr)+F^m
\ \ \text{on}\ \ \omega_{h,\infty},
\label{e2}\\[1mm]
 i\hbar\,\frac{\Psi^m-\widetilde{\Psi}^m}{\tau_m/2}
 =\Delta V\frac{\Psi^m+\widetilde{\Psi}^m}{2}\ \
\ \ \text{on}\ \ \omega_{h,\infty},
\label{e3}
\end{gather}
with the boundary and initial conditions
\begin{gather}
 \breve{\Psi}^m|_{\Gamma_{h,\infty}}=0,\ \ \widetilde{\Psi}^m|_{\Gamma_{h,\infty}}=0,\ \
 \Psi^m|_{\Gamma_{h,\infty}}=0,
\label{e4}\\[1mm]
 \Psi^0=\Psi^0_{h}\ \ \text{on}\ \ \overline{\omega}_{h,\infty},
\label{e5}
\end{gather}
for any $m\geq 1$, where $\Delta V:=V-\widetilde{V}$ and the auxiliary 1D potential $\widetilde{V}=\widetilde{V}(x_1)$ satisfies $\widetilde{V}(x_1)=V_\infty$ for $x_1\geq X_0$. In the simplest case, $\widetilde{V}(x)=V_\infty$ (but a non-constant $\widetilde{V}$ is necessary when extending the results to the case of an infinite parallelepiped with different limit values $V_{\pm\infty}$ of $V(x)$ as $x_1\to\pm\infty$).
\par We have added the free term $F^m$ into (\ref{e2}) to study stability in more detail below.
\par The construction of this splitting is similar to the case of the 2D schemes without averages \cite{DZZ13} and with the Numerov average \cite{ZR13}.
Note that we have omitted operators $\overline{s}_N$ arising in the course of splitting on both sides of (\ref{e1}) and (\ref{e3}).
Clearly equations \eqref{e1} and \eqref{e3} are simply reduced to the explicit formulas
\begin{equation}
 \breve{\Psi}^m=\mathcal{E}^m\Psi^{m-1},\ \
 \Psi^m=\mathcal{E}^m\widetilde{\Psi}^m,\ \ \text{with}\ \
 \mathcal{E}^m
 :=\Bigl({1-i\displaystyle{\frac{\tau_m}{4\hbar}}\, \Delta V}\Bigr)/
        \Bigl({1+i\displaystyle{\frac{\tau_m}{4\hbar}}\, \Delta V}\Bigr).
\label{e6}
\end{equation}
The main equation (\ref{e2}) is similar to the original one (\ref{f1}) on the time level $m$ but it is essentially simplified by
replacing $V(x)$ by $\widetilde{V}(x_1)$. The functions $\breve{\Psi}$ and $\widetilde{\Psi}$ are auxiliary unknowns while $\Psi$ is the main one.
\par From (\ref{e6}) and (\ref {e4}) we get immediately
\begin{equation}
 |\breve{\Psi}^m|=|\Psi^{m-1}|,\ \ |\Psi^m|=|\widetilde{\Psi}^m|\ \ \mbox{on}\ \ \overline{\omega}_{h,\infty};
\label{ff1}
\end{equation}
moreover, since $\Delta V_{{\bf j}}=0$ for $j_1\geq J_1-1$, we simply have
\begin{equation}
 \breve{\Psi}^m_{{\bf j}}=\Psi^{m-1}_{{\bf j}},\ \ \Psi^m_{{\bf j}}=\widetilde{\Psi}^m_{{\bf j}}\ \ \mbox{for}\
 \ j_1\geq J_1-1.
\label{ff2}
\end{equation}
\par This splitting modifies the scheme (\ref{f1}), (\ref{f2}) only in time and is symmetric in time due to steps (\ref{e1}) and (\ref{e3}). Thus, concerning the approximation error, it reduces neither the 4th order in $|h|$ nor the 2nd order in $\tau_{\max}$. This can be checked also more formally similarly to \cite{DZZ13,ZR13}.
\section{Stability of the splitting higher order scheme on an infinite space mesh}
Let $H_h$ be a Hilbert space of mesh functions $W$: $\overline{\omega}_{h,\infty}\to {\mathbb C}$ such that
$\left. W\right|_{\Gamma_{h,\infty}}=0$ and
\[
\sum_{j_1=1}^{\infty} \sum_{j_2=1}^{J_2-1}\ldots \sum_{j_n=1}^{J_n-1}\left|W_{{\mathbf j}}\right|^2<\infty
\]
endowed with the following mesh counterpart of the inner product in $L^2(\Pi_{\infty})$
\[
\left(U,W\right)_{H_h}
:=\sum_{j_1=1}^{\infty} \sum_{j_2=1}^{J_2-1}\ldots \sum_{j_n=1}^{J_n-1}U_{{\mathbf j}}W_{{\mathbf j}}^* h_1\dots h_n.
\]
\par We only need the first assumption (\ref{a4}) in all this section.
\begin{proposition}
\label{ph}
Let $\Psi_{h}^0, F^m\in H_{h}$ for any $m\geq 1$.
Then there exists a unique solution to the splitting scheme \eqref{e1}-\eqref{e5}
such that $\Psi^m\in H_{h}$ for any $m\geq 0$, and the following $L^2$-stability bound holds
\begin{equation}
 \max_{0\leq m\leq M}\|\Psi^m\|_{H_{h}}
 \leq\|\Psi^0_{h}\|_{H_{h}}
 +\frac{2}{\hbar}\left(\frac{3}{2}\right)^n\sum_{m=1}^M
 \left\|F^m\right\|_{H_{h}}\tau_m\ \ \text{for any}\ \ M\geq 1.
\label{h2}
\end{equation}
\par Moreover, in the particular case $F=0$, the following mass conservation law holds
\begin{equation}
 \|\Psi^m\|_{H_{h}}^2=
 \|\Psi^0_{h}\|_{H_{h}}^2
 \ \ \text{for any}\ \ m\geq 1.
\label{h3}
\end{equation}
\end{proposition}
\begin{proof}
We first rewrite the main equation (\ref{e2}) as a suitable operator equation in $H_h$. We set $\Lambda_kW:=-\partial_k\bar{\partial}_kW$
on $\omega_{h,\infty}$ and $\Lambda_kW:=0$ on $\Gamma_{h,\infty}$, for $1\leq k\leq n$. Then
\begin{gather*}
\Lambda_k,\ \
s_{Nk}=I-\frac{h_k^2}{12}\,\Lambda_k,\ \
\overline{s}_N=s_{N1}\dots s_{Nn},\ \
\overline{s}_{N\widehat{k}}=\prod_{1\leq \ell\leq n,\, \ell\neq k}s_{N\ell}
\end{gather*}
and $-\bar{\Delta}_{hN}=s_{N1}\Lambda_1+\dots + s_{Nn}\Lambda_n$ are bounded self-adjoint operators in $H_h$. Moreover
\[
\left(s_{Nk}W,W\right)_{H_h}\geq \frac{2}{3} \left\|W\right\|_{H_h}^2\ \ \mbox{for any}\ W\in H_h
\]
(for $k=1$, see \cite{DZZ09} taking there the particular space average with a parameter $\theta=\frac{1}{12}$), therefore the inverse operator $s_{Nk}^{-1}$ exists and is bounded
\begin{equation}
 \left\|s_{Nk}^{-1}\right\|_{{\mathcal L}(H_h)}\leq \frac{3}{2}.
\label{i1}
\end{equation}
Therefore we can consider (\ref{e2}) as an operator equation in $H_h$. In the spirit of \cite{SA08}, we apply $\overline{s}_{N}^{-1}=s_{N1}^{-1}\cdots s_{Nn}^{-1}$ to it and obtain
\begin{equation}
i\hbar\,\frac{\widetilde{\Psi}^m-\breve{\Psi}^m}{\tau_m}
 =A_h\,\frac{\widetilde{\Psi}^m+\breve{\Psi}^m}{2}+\overline{s}_{N}^{-1}F^m
\ \ \text{in}\ \ H_h,
\label{i2}
\end{equation}
where $A_h:=c_{\hbar}\left(
s_{N1}^{-1}\Lambda_1 +\dots +s_{Nn}^{-1}\Lambda_n\right)+\widetilde{V}I$.
Since $\Lambda_k$ and $s_{Nk}$ commute, so do  $\Lambda_k$ and $s_{Nk}^{-1}$, and consequently $A_h$ is a bounded self-adjoint operator in $H_h$.
\par We rewrite equation (\ref{i2}) in another form
\[
\left(I+i\frac{\tau_m}{2\hbar}\, A_h\right)\widetilde{\Psi}^m
=B^m
:=
\left(I-i\frac{\tau_m}{2\hbar}\, A_h\right)\breve{\Psi}^m
-i\frac{\tau_m}{\hbar}\,\overline{s}_{N}^{-1}F^m.
\]
Since the operator $I+i\frac{\tau_m}{2\hbar}\, A_h$ is invertible, the equation has a unique solution $\widetilde{\Psi}^m\in H_h$ provided that $\breve{\Psi}^m,F^m\in H_h$. This implies the existence of a unique solution of the splitting scheme such that $\Psi^m\in H_h$ for any $m\geq 0$.
\par We can now follow the lines of \cite{DZZ13}. Note first that the pointwise equalities (\ref{ff1}) imply
\begin{equation}
 \|\breve{\Psi}^m\|_{H_h}=\|\Psi^{m-1}\|_{H_h},\ \ \|\Psi^m\|_{H_h}=\|\widetilde{\Psi}^m\|_{H_h}.
\label{j1}
\end{equation}
Multiplying the operator equation (\ref{i2}) by $\frac{\widetilde{\Psi}^m+\breve{\Psi}^m}{2}$, separating the imaginary part of the result and using the property $A_h=A_h^*$, we get
\[
\frac{\hbar}{2\tau_m} \left(
\|\widetilde{\Psi}^m\|_{H_h}^2-\|\breve{\Psi}^m\|_{H_h}^2
\right)
=
\Ima
\Bigl(
 \overline{s}_{N}^{-1}F^m,
\frac{\widetilde{\Psi}^m+\breve{\Psi}^m}{2}
\Bigr)_{H_h}.
\]
Applying equalities (\ref{j1}), multiplying both sides by $\frac{2\tau_m}{\hbar}$ and summing up the result over $m=1,\dots,M$, we obtain
\begin{equation}
\|\Psi^M\|_{H_h}^2
=\|\Psi^0_h\|_{H_h}^2
+\frac{2}{\hbar}
\sum_{m=1}^M
\Ima
\Bigl(
 \overline{s}_{N}^{-1}F^m,
\frac{\widetilde{\Psi}^m+\breve{\Psi}^m}{2}
\Bigr)_{H_h} \tau_m.
\label{j2}
\end{equation}
Owing to (\ref{i1}) and (\ref{j1}) we have
\begin{gather*}
\|\Psi^m\|_{H_h}^2
\leq
\|\Psi^0\|_{H_h}^2
+\frac{2}{\hbar}
\sum_{m=1}^M
\|s_{N}^{-1}F^m\|_{H_h}
\frac{1}{2}\,(\|\widetilde{\Psi}^m\|_{H_h}+\|\breve{\Psi}^m\|_{H_h}) \tau_m
\\
\leq
\|\Psi^0\|_{H_h}^2
+\frac{2}{\hbar}
\left(\frac{3}{2}\right)^n
\sum_{m=1}^M
\|F^m\|_{H_h} \tau_m
\max_{0\leq m\leq M}\|\Psi^m\|_{H_h}.
\end{gather*}
This inequality directly implies bound (\ref{h2}).
Also (\ref{h3}) follows from (\ref{j2}).
\end{proof}

\section{The splitting higher order scheme on a finite space mesh}
\setcounter{equation}{0}
\setcounter{proposition}{0}
\setcounter{theorem}{0}
\setcounter{lemma}{0}
\setcounter{corollary}{0}
\setcounter{remark}{0}
The splitting scheme (\ref{e1})-(\ref{e5}) is not practically implementable because of the infinite number of unknowns on each time level.
We now intend to restrict its solution to a finite space mesh $\overline{\omega}_h:=\{ x_{{\bf j}}\in \overline{\omega}_{h,\infty};\, 0\leq j_1\leq J_1\}$.
Let $\omega_h:=\{ x_{{\bf j}}\in \omega_{h,\infty};\, 1\leq j_1\leq J_1-1\}$ and
$\partial\omega_h=\overline{\omega}_h\backslash \omega_h$
be its internal part and boundary, and
$\Gamma_{1h}:=\{x_{{\bf j}};\, j_1=J_1,\, 1\leq j_2\leq J_2-1,\dots,1\leq j_n\leq J_n-1\}$ and
$\Gamma_h=\partial\omega_h\backslash \Gamma_{1h}$ be the boundary parts.
Let $\overline{\omega}_{h\widehat{1}}$
and $\omega_{h\widehat{1}}$ be $(n-1)$-dimensional versions of $\overline{\omega}_h$ and $\omega_h$ (excluding the direction $x_1$)
as well as $\omega_{h1}:=\{j_1h_1;\, 1\leq j_1\leq J_1-1\}$ (so that $\omega_h=\omega_{h1}\times\omega_{h\widehat{1}}$).
\par
By definition, \textit{the discrete transparent boundary condition} (TBC) is a boundary condition on $\Gamma_{1h}$ which admits to accomplish the above mentioned restriction.
\par To write down the discrete TBC, we need operators
\[
s_{N1}^{\pm}W_j=\frac{5}{12}\,W_j+\frac{1}{12}\,W_{j\pm 1},\ \ \overline{s}_{N,\widehat{1k}}:=\prod_{2\leq \ell\leq n,\, \ell\neq k} s_{N\ell},
\]
so that $s_{N1}=s_{N1}^-+s_{N1}^+$ and $\overline{s}_{N\widehat{k}}=s_{N1}\overline{s}_{N\widehat{1k}}$, for $2\leq k\leq n$.
\par We also exploit the direct and inverse discrete Fourier sine transforms in direction $x_k$
\begin{gather*}
P^{(q)}=\left({\mathcal F}_kP\right)^{(q)}:= \frac{2}{J_k} \sum_{j=1}^{J_k-1} P_j \sin \frac{\pi q j}{J_k},\ \ 1\leq q\leq J_k-1,
\\
P_j=\left({\mathcal F}_k^{-1}P^{(\cdot)}\right)_j:= \sum_{q=1}^{J_k-1} P^{(q)} \sin \frac{\pi q j}{J_k},\ \ 1\leq j\leq J_k-1.
\end{gather*}
The corresponding eigenvalues of $-\partial_k\bar{\partial}_k$ and $s_{Nk}$ are
$\lambda_q^{(k)}$ and $\sigma_q^{(k)}=1-\frac{1}{3}\sin^2 \frac{\pi q h_k}{2X_k}\in (\frac23,1)$.
\par Given a function $W$: $\overline{\omega}_{h}\to {\mathbb C}$, denote by $W_{J_1}$ its trace on $\Gamma_{1h}$.
Let ${\bf \Psi}^m_{J_1}=\{\Psi_{J_1}^0,\dots,\Psi_{J_1}^m\}$ be the vector function. Given functions
$R,Q$: $\overline{\omega}^{\,\tau}\to {\mathbb C}$, we denote by
\[
(R*Q)^m:=\sum_{p=0}^m R^pQ^{m-p},\ \ m\geq 0,
\]
their discrete convolution product.
Let the time mesh be uniform with a step $\tau>0$ below.
\begin{proposition}
\label{pk}
Let $ F^m=0$ and $\Psi_{h}^0=0$ on
$\omega_{h,\infty}\backslash\omega_h$ for any $m\geq 1$ and
$\bigl.\Psi_{h}^0\bigr|_{j_1=J_1-1}=0$.
\par The solution to the splitting scheme \eqref{e1}-\eqref{e5}
such that $\Psi^m\in H_{h}$ for any $m\geq 0$ satisfies the following three-step splitting scheme on the finite space mesh $\overline{\omega}_h$
\begin{gather}
 i\hbar\, \frac{\breve{\Psi}^m-\Psi^{m-1}}{\tau/2}
 =\Delta V\frac{\breve{\Psi}^m+\Psi^{m-1}}{2}\ \
 \text{on}\ \ \omega_{h}\cup \Gamma_{1h},
\label{l1}\\[1mm]
 i\hbar\overline{s}_N\,\frac{\widetilde{\Psi}^m-\breve{\Psi}^m}{\tau}
 =-c_{\hbar}\bar{\Delta}_{hN}\frac{\widetilde{\Psi}^m+\breve{\Psi}^m}{2}
 +\overline{s}_N\Bigl(\widetilde{V}\frac{\widetilde{\Psi}^m+\breve{\Psi}^m}{2}\Bigr)
+F^m
\ \ \text{on}\ \ \omega_{h},
\label{l2}\\[1mm]
 i\hbar\,\frac{\Psi^m-\widetilde{\Psi}^m}{\tau/2}
 =\Delta V\frac{\Psi^m+\widetilde{\Psi}^m}{2}\ \
 \text{on}\ \ \omega_{h}\cup \Gamma_{1h},
\label{l3}
\end{gather}
with the boundary and initial conditions
\begin{equation}
 \breve{\Psi}^m|_{\Gamma_h}=0,\ \ \widetilde{\Psi}^m|_{\Gamma_h}=0,\ \
 \Psi^m|_{\Gamma_h}=0,
\label{l4}
\end{equation}
\[
 {\mathcal D}_{1h}(\widetilde{\Psi}^m,\breve{\Psi}^m)
:= c_{\hbar} \overline{s}_{N\widehat{1}}\bar{\partial}_1\, \frac{\widetilde{\Psi}^m+\breve{\Psi}^m}{2}
-h_1s_{N1}^-
\Bigl\{i\hbar\overline{s}_{N\widehat{1}}\,\frac{\widetilde{\Psi}^m-\breve{\Psi}^m}{\tau}
\]
\begin{equation}
+\left[c_{\hbar}
\left(\overline{s}_{N\widehat{12}}\partial_2\bar{\partial}_2
+\dots+\overline{s}_{N\widehat{1n}}\partial_n\bar{\partial}_n\right)
-V_{\infty}\overline{s}_{N\widehat{1}}
\right]\frac{\widetilde{\Psi}^m+\breve{\Psi}^m}{2}
\Bigr\}
=c_{\hbar}{\mathcal S}^m_{\rm ref}\widetilde{{\mathbf\Psi}}^m_{J_1}\ \ \mbox{on}\ \ \Gamma_{1h},
\label{l5}
\end{equation}
\begin{equation}
\Psi^0=\Psi^0_h\ \ \text{on}\ \ \overline{\omega}_{h}  ,
\label{l6}
\end{equation}
for any $m\geq 1$.
\par The operator on the right in the discrete TBC (\ref{l5}) is given by
\begin{equation}
{\mathcal S}^m_{\rm ref} {\mathbf \Phi}^m:=
{\mathcal F}_2^{-1}\dots {\mathcal F}_n^{-1}\left[
\sigma^{(2)}_{q_2}\dots \sigma^{(n)}_{q_n}R_{\mathbf q}*\Phi^{\mathbf q}\right]^m,
\label{l7}
\end{equation}
for any $\Phi$: $\omega_{h\widehat{1}}\times \overline{\omega}^{\,\tau}\to {\mathbb C}$ such that $\Phi^0=0$, with
${\mathbf \Phi}^m:=\{\Phi^0,\dots,\Phi^m\}$ and
\[
 \Phi^{\mathbf q}
 :=({\mathcal F}_n\dots ({\mathcal F}_2\Phi)^{(q_2)}\dots)^{(q_n)},\ \
 {\mathbf q}=(q_2,\dots,q_n).
\]
The discrete convolution kernel in (\ref{l7}) has the form
\begin{equation}
R_{\mathbf q}=R\left[V_{\infty,{\mathbf q}}\right],\ \
\text{with}\ \ V_{\infty,{\mathbf q}}=
V_{\infty}+c_{\hbar}
\Bigl(
\frac{\lambda^{(2)}_{q_2}}{\sigma^{(2)}_{q_2}}+\dots +\frac{\lambda^{(n)}_{q_n}}{\sigma^{(n)}_{q_n}}
\Bigr),
\label{l8}
\end{equation}
where $R[V_{\infty}]$ can be computed recurrently by
\begin{gather}
R^0[V_{\infty}]=c_1,\ \
R^1[V_{\infty}]=-c_1\varkappa \mu,
\label{l8a}\\[1mm]
R^m[V_{\infty}]=\frac{2m-3}{m}\, \varkappa \mu R^{m-1}[V_{\infty}]-\frac{m-3}{m}\, \varkappa^2 R^{m-2}[V_{\infty}]\ \ \mbox{for}\ \ m\geq 2.
\label{l8b}
\end{gather}
Here the coefficients $c_1,\varkappa$ and $\mu$ are defined by
\[
c_1=-\frac{|\alpha|^{1/2}}{2}e^{-i(\arg \alpha)/2},\ \ \varkappa=-e^{i\arg \alpha},\ \ \mu=\frac{\beta}{|\alpha|}\in(-1,1),
\]
\[
\alpha=2a+\frac{2}{3}h_1^2a^2\neq 0,\ \ \arg \alpha\in(0,2\pi),\ \ \beta=2\Rea a+\frac{2}{3}h_1^2|a|^2,\ \
a=\frac{V_{\infty}}{2 c_\hbar}+i\, \frac{\hbar}{\tau c_\hbar}.
\]
\end{proposition}
\begin{proof}
Clearly it is sufficient to derive the discrete TBC (\ref{l5}) for the solution of the splitting scheme (\ref{e1})-(\ref{e5}) under the above assumptions on $\Psi^0_h$ and $F$.
\par
Due to property (\ref{f2}), equations (\ref{e1})-(\ref{e3}) are reduced on $\omega_{h,\infty}\backslash \omega_h$ and for $m\geq 1$ to the equation
\begin{equation}
i\hbar \overline{s}_N\bar{\partial}_t \Psi
=\left(-c_{\hbar}\bar{\Delta}_{hN}+V_{\infty}\overline{s}_N\right)\overline{s}_t \Psi
\ \ \mbox{on}\ \ (\omega_{h,\infty}\backslash \omega_h)\times\omega^{\tau}.
\label{n1}
\end{equation}
Also the boundary and initial conditions (\ref{l4}) and (\ref{l6}) imply that
\begin{equation}
 \left. \Psi^m\right|_{\Gamma_{h,\infty}\backslash \Gamma_h}=0\ \ \mbox{for}\ \ m\geq 1,\ \
\Psi^0=0\ \ \mbox{on}\ \ \{ j_1h_1 \}^{\infty}_{j_1=J_1-1}\times\overline{\omega}_{h\widehat{1}}
\label{n2}
\end{equation}
The discrete TBC (\ref{l5}) takes the following form
\[
c_{\hbar} \overline{s}_{N\widehat{1}}\bar{\partial}_1 \overline{s}_t \Psi^m
-h_1 s_{N1}^-
\left\{
i\hbar\overline{s}_{N\widehat{1}} \bar{\partial}_t \Psi^m
+\left[
c_{\hbar}
\left(
\overline{s}_{N\widehat{12}}\partial_2\bar{\partial}_2
+\dots+
\overline{s}_{N\widehat{1n}}\partial_n\bar{\partial}_n
\right)
-V_{\infty}\overline{s}_{N\widehat{1}}
\right]\overline{s}_t \Psi^m
\right\}
\]
\begin{equation}
=
c_{\hbar}{\mathcal S}^m_{\rm ref} {\mathbf\Psi}^m_{J_1}
\ \ \mbox{on}\ \ \Gamma_{1h}.
\label{n3}
\end{equation}
\par Similarly to \cite{DZ06,DZZ09,IZ11}, we first construct the discrete TBC in the following symmetric form with respect to $x_1$
\[
\overset{\circ}{\partial}_1
\Bigl\{
c_{\hbar} \overline{s}_{N\widehat{1}} \overline{s}_t \Psi^m
+\frac{h_1^2}{12}\,
\left[
i\hbar \overline{s}_{N\widehat{1}}\bar{\partial}_t \Psi^m
+\left[
c_{\hbar}
\left(
\overline{s}_{N\widehat{12}}\partial_2\bar{\partial}_2
+\dots+
\overline{s}_{N\widehat{1n}}\partial_n\bar{\partial}_n
\right)
-V_{\infty}\overline{s}_{N\widehat{1}}
\right]
\overline{s}_t \Psi^m
\right]
\Bigr\}
\]
\begin{equation}
=
c_{\hbar}{\mathcal S}^m_{\rm ref} {\mathbf \Psi}^m_{J_1}
\ \ \mbox{on}\ \ \Gamma_{1h},\ \ \text{for any}\ \ m\geq 1.
\label{n4}
\end{equation}
Using the elementary formulas $\overset{\circ}{\partial}_1=\bar{\partial}_1+\frac{h_1}{2}\partial_1\bar{\partial}_1$ and
$\frac{h_1^2}{12}\,\overset{\circ}{\partial}_1=-h_1 s_{N1}^-+\frac{h_1}{2}s_{N1}$
and equation (\ref{n1}) on $\Gamma_{1h}$, we see that the discrete TBC (\ref{n3}) is equivalent to (\ref{n4}).

Now following \cite{AES03,DZ06,DZZ13,IZ11}, we apply the operator ${\mathcal F}_2\dots{\mathcal F}_n$ to equations
(\ref{n1}) and (\ref{n3}). Dividing the result by $\sigma^{(2)}_{q_2}\dots \sigma^{(n)}_{q_n}$, we obtain that a function $P:=\left({\mathcal F}_2\dots{\mathcal F}_n \Psi\right)^{\mathbf q}$ satisfies the 1D Numerov-Crank-Nicolson scheme
for the 1D Schr\"odinger equation with constant coefficients
\begin{equation}
i\hbar s_{N1}P=
\left(
-c_{\hbar}\partial_1\bar{\partial}_1+V_{\infty,{\mathbf q}}s_{N1}
\right)
\overline{s}_t P\ \ \mbox{on}\ \ \{ jh_1 \}^{\infty}_{j=J_1} \times \omega^{\tau},
\label{n5}
\end{equation}
with zero initial data
\begin{equation}
P^0=0\ \ \mbox{on}\ \ \{ j_1h_1 \}^{\infty}_{j_1=J_1-1},
\label{n6}
\end{equation}
see (\ref{n2}), and the boundary condition, for any $m\geq 1$
\begin{equation}
\overset{\circ}{\partial}_1
\left.\left[
c_{\hbar} \overline{s}_t P^m
+\frac{h_1^2}{12}
\left(
i\hbar \bar{\partial}_t P^m
-V_{\infty,{\mathbf q}}\overline{s}_t P^m
 \right)
\right]
\right|_{j=J_1}
=\frac{c_{\hbar}}{\sigma^{(2)}_{q_2}\dots \sigma^{(n)}_{q_n}}
\left(
{\mathcal F}_2\dots{\mathcal F}_n {\mathcal S}^m_{\rm ref}
{\mathbf\Psi}^m_{J_1}\right)^{\mathbf q}.
\label{n7}
\end{equation}
\par As it was calculated in \cite{DZZ09} (taking there $\theta=\frac{1}{12}$), solutions of (\ref{n5}), (\ref{n6}) satisfy
\begin{equation}
\overset{\circ}{\partial}_1
\Bigl.\Bigl[
c_{\hbar} \overline{s}_t P^m
+\frac{h_1^2}{12}
\left(
i\hbar \bar{\partial}_t P^m
-V_{\infty,{\mathbf q}}\overline{s}_t P^m
 \right)
\Bigr]
\Bigr|_{j=J_1}
=c_{\hbar}
\left(
R\left[V_{\infty,{\mathbf q}}\right]*P_{J_1}\right)^m,
\label{n8}
\end{equation}
for any $m\geq 1$, where $R\left[V_{\infty}\right]$ can be computed by the recurrent relations (\ref{l8a}), (\ref{l8b}). (Actually the formulas from \cite{DZZ09} are slightly modified and refined from misprints; also the recently checked fixed sign in the formula for $c_{1}$ is taken into account.) This was done in \cite{DZZ09} under suitable conditions on $P$ valid here due to the stability bound (\ref{h2}).
\par Comparing (\ref{n7}) and (\ref{n8}) leads to
\[
 \frac{1}{\sigma^{(2)}_{q_2}\dots \sigma^{(n)}_{q_n}}
\left(
{\mathcal F}_2\dots{\mathcal F}_n {\mathcal S}^m_{\rm ref}
{\mathbf\Psi}^m_{J_1}\right)^{\mathbf q}
=\left(
R\left[V_{\infty,{\mathbf q}}\right]*P\right)^m,
\]
and after multiplying by $\sigma^{(2)}_{q_2}\dots \sigma^{(n)}_{q_n}$ and applying
${\mathcal F}_2^{-1}\dots{\mathcal F}_n^{-1}$, we get formula (\ref{l7}).
\end{proof}
The form of the discrete TBC follows our previous studies \cite{DZ06,DZ07,DZZ09,DZZ13}
allowing to ensure both stability of schemes and the stable numerical implementation of the discrete TBCs; moreover, for $n=2$ they are equivalent to those constructed in \cite{ZZ11,IZ11} in the particular case $\theta=\frac{1}{12}$.
Notice that the following important summation identity coupling the operators in the main equation (\ref{l2}) and the discrete TBC (\ref{l5}) holds
\[
\Bigl(
 i\hbar \overline{s}_N\, \frac{\widetilde{\Psi}^m-\breve{\Psi}^m}{\tau}
 +
c_{\hbar}\bar{\Delta}_{hN}\frac{\widetilde{\Psi}^m+\breve{\Psi}^m}{2}
-\overline{s}_N\Bigl(\widetilde{V}\frac{\widetilde{\Psi}^m+\breve{\Psi}^m}{2}\Bigr),W\Bigr)_{\omega_h}
-\Bigl(
{\mathcal D}_{1h}(\widetilde{\Psi}^m,\breve{\Psi}^m)
_{J_1},W_{J_1}
\Bigr)_{\omega_{h\widehat{1}}}
\]
\[
=
\Bigl(
 \overline{s}_{N\widehat{1}}
\Bigl(
i\hbar\,\frac{\widetilde{\Psi}^m-\breve{\Psi}^m}{\tau}
-\widetilde{V}
\frac{\widetilde{\Psi}^m+\breve{\Psi}^m}{2}
\Bigr),W
\Bigr)_{\overline{\omega}_{hN1}\times\omega_{h\widehat{1}}}
+
c_{\hbar}
 \Bigl(
\overline{s}_{N\widehat{1}}\bar{\partial}_1\, \frac{\widetilde{\Psi}^m+\breve{\Psi}^m}{2}, \bar{\partial}_1 W\Bigr)_{\tilde{\omega}_h}
\]
\begin{equation}
+
c_{\hbar}
 \Bigl(
-\Bigl(
\overline{s}_{N\widehat{12}}\partial_2\bar{\partial}_2+\dots +\overline{s}_{N\widehat{1n}}\partial_n\bar{\partial}_n
\Bigr)
\frac{\widetilde{\Psi}^m+\breve{\Psi}^m}{2},W
\Bigr)_{\overline{\omega}_{hN1}\times\omega_{h\widehat{1}}}
\label{m1}
\end{equation}
for any $W$: $\overline{\omega}_h\to{\mathbb C}$ such that $\left. W\right|_{j_1=0}=0$.
Here we have used the collection of $L^2$-mesh inner products
\[
\left(U,W\right)_{\omega_h}
:=\sum_{j_1=1}^{J_1-1}\ldots\sum_{j_n=1}^{J_n-1}
U_{{\mathbf j}}W^*_{{\mathbf j}}h_1\ldots h_n,\,\
\left(U,W\right)_{\omega_{h\widehat{1}}}
:=\sum_{j_2=1}^{J_2-1}\dots\sum_{j_n=1}^{J_n-1}
U_{j_2,\dots,j_n}W^*_{j_2,\dots,j_n}h_2\dots h_n,
\]
\[
\left(U,W\right)_{\widetilde\omega_h}:=\left(U,W\right)_{\omega_h}+\left(U_{J_1},W_{J_1}\right)_{\omega_{h\widehat{1}}}h_1,
\]
\begin{equation}
\left(U,W\right)_{\overline{\omega}_{hN1}\times\omega_{h\widehat{1}}}
:=\left(s_{N1}U,W\right)_{\omega_h}
+\left(s_{N1}^-U_{J_1},W_{J_1}\right)_{\omega_{h\widehat{1}}} h_1.
\label{m2}
\end{equation}
According to \cite{DZZ09} (taking there $\theta=\frac{1}{12}$), the sesquilinear form (\ref{m2}) is Hermitian and positive definite on functions
$U,W$: $\overline{\omega}_h\to{\mathbb C}$ such that $\left. U\right|_{\Gamma_h}=\left. W\right|_{\Gamma_h}=0$.
In what follows, we need the norms $\|\cdot\|_{\omega_h}$ and $\|\cdot\|_{\widetilde \omega_h}$ associated to the first and third of these inner products.
\par
The summation identity (\ref{m1}) appears after rearranging terms on its left-hand side and summing by parts with respect to $x_1$ in the term $c_{\hbar}\overline{s}_{N\widehat{1}}\partial_1\bar{\partial}_1$.
\begin{lemma}
\label{ln}
The operator ${\mathcal S}^m_{\rm ref}$ satisfies the inequality \cite{DZ06}
\begin{equation}
\Ima\sum_{m=1}^M \left({\mathcal S}^m_{\rm ref}{\mathbf \Phi}^m,\overline{s}_t \Phi^m\right)_{\omega_{h\widehat{1}}}\tau \geq 0\ \ \mbox{for any}\ \ M\geq 1,
\label{nn1}
\end{equation}
for any function $\Phi$: $\omega_{h\widehat{1}}\times  \overline{\omega}^{\,\tau}\to {\mathbb C}$ such that $\Phi^0=0$.
\end{lemma}
\begin{proof}
Following \cite{DZ06}, we use formula (\ref{l7}) and standard properties of ${\mathcal F}_2,\dots,{\mathcal F}_n$ and get
\[
\left({\mathcal S}^m_{\rm ref}{\mathbf \Phi}^m,\overline{s}_t \Phi^m\right)_{\omega_{h\widehat{1}}}
=\frac{X_2\dots X_n}{2^{n-1}}
\sum_{q_2=1}^{J_2-1}\ldots\sum_{q_n=1}^{J_n-1}
\sigma^{(2)}_{q_2}\dots \sigma^{(n)}_{q_n}
\left(
R_{\mathbf q}*\Phi^{\mathbf q}
\right)^m
\left(\overline{s}_t \Phi^m\right)^*.
\]
Consequently
\begin{gather}
\Ima \sum_{m=1}^M \left({\mathcal S}^m_{\rm ref}{\mathbf \Phi}^m,\overline{s}_t \Phi^m\right)_{\omega_{h\widehat{1}}}\tau
\nonumber\\
=\frac{X_2\dots X_n}{2^{n-1}}
\sum_{q_2=1}^{J_2-1}\ldots\sum_{q_n=1}^{J_n-1}
\sigma^{(2)}_{q_2}\dots \sigma^{(n)}_{q_n}
\Ima \sum_{m=1}^M
\left(
R[V_{\infty {\mathbf q}}]*\Phi^{\mathbf q}
\right)^m
\left(\overline{s}_t \Phi^m\right)^*\tau.
\label{nn2}
\end{gather}
The result follows from the similar 1D inequality proved in \cite{DZZ09} (taking there $\theta=\frac{1}{12}$).
\end{proof}
By construction, the splitting scheme (\ref{l1})-(\ref{l6}) on the finite space mesh has a solution. Let us prove its uniqueness;
notice that we do not need any restrictions on $\tau$ to this end (in contrast to \cite{ZR13}).
Let $\bigl.\Psi^0_h\bigr|_{j_1=J_1-1,\,J_1}=0$ below.
\begin{proposition}
\label{po}
The solution of the splitting scheme (\ref{l1})-(\ref{l6}) on the finite space mesh is unique.
It satisfies the following $L^2$-stability bound
\begin{equation}
\max_{0\leq m\leq M}\|\Psi^m\|_{\widetilde\omega_h}\leq \|\Psi^0_h\|_{\widetilde\omega_h}
+\frac{2}{\hbar}\,\left(\frac{3}{2}\right)^n \sum_{m=1}^M \| F^m\|_{\omega_h}\tau\ \ \mbox{for any}\ \ M\geq 1.
\label{o1}
\end{equation}
\end{proposition}
\begin{proof}
Assume that there exist two solutions of the scheme (\ref{l1})-(\ref{l6}) and denote by $Y$ their difference.
Clearly $Y$ satisfies the homogeneous scheme (\ref{l1})-(\ref{l6}), with $F=0$ and $\Psi^0_h=0$.
\par
In order to establish uniqueness, it is sufficient to prove that if $Y^0=0,\dots,Y^{m-1}=0$, then $Y^m=0$. Under this assumption $Y^m$ satisfies a homogeneous equation
\begin{equation}
i\frac{\hbar}{\tau}\, \overline{s}_N Y^m
=-\frac{c_{\hbar}}{2}\, \bar{\Delta}_{hN}Y^m
+\frac{1}{2}\,\overline{s}_N(\widetilde{V}Y^m)\ \ \mbox{on}\ \ \omega_h,
\label{o2}
\end{equation}
together with the homogeneous boundary conditions
\begin{equation}
\left. Y^m\right|_{\Gamma_h}=0,\ \ {\mathcal D}_{1h}\left(Y^m,0\right)=c_{\hbar}{\mathcal S}^m_{\rm ref}{\mathbf Y}^m_{J_1}\ \ \mbox{on}\ \ \Gamma_{1h},
\label{o3}
\end{equation}
where ${\mathbf Y}^m=\{0,\dots,0,Y^m\}$, with $0$ appearing $m$ times.
\par Following \cite{ZR13}, applying the summation identity (\ref{m1}) in the case $\widetilde{\Psi}^m=Y^m$, $\breve{\Psi}^m=0$ and $W=Y^m$, and using (\ref{o2}) and (\ref{o3}), we get
\[
\Bigl(
\overline{s}_{N\widehat{1}}
\Bigl(
i\frac{\hbar}{\tau} Y^m
-\frac12 \widetilde{V}Y^m\Bigr),Y^m
\Bigr)_{\overline{\omega}_{hN1}\times\omega_{h\widehat{1}}}
+
\frac{c_{\hbar}}{2}
\Bigl(
\overline{s}_{N\widehat{1}}\, \bar{\partial}_1Y^m,\bar{\partial}_1Y^m
\Bigr)_{\widetilde{\omega}_{h}}
\]
\begin{equation}
+
\frac{c_{\hbar}}{2}
 \left(
-\left(
\overline{s}_{N\widehat{12}}\,\partial_2\bar{\partial}_2+\dots +\overline{s}_{N\widehat{1n}}\,\partial_n\bar{\partial}_n
\right)Y^m
,Y^m
\right)_{\overline{\omega}_{hN1}\times\omega_{h\widehat{1}}}
+
c_{\hbar}
\left(
{\mathcal S}^m_{\rm ref}{\mathbf Y}^m_{J_1},Y^m_{J_1}
\right)_{\omega_{h\widehat{1}}}=0.
\label{q1}
\end{equation}
\par Let $H_{h\widehat{1}}$ be the space of functions $P$: $\overline{\omega}_{h\widehat{1}}\to{\mathbb C}$ such that $P=0$ on
$\partial\omega_{h\widehat{1}}=\overline{\omega}_{h\widehat{1}}\backslash\omega_{h\widehat{1}}$ endowed with the inner product $\left(\cdot,\cdot\right)_{\omega_{h\widehat{1}}}$.
Setting $AP=0$ on $\partial\omega_{h\widehat{1}}$ for
 $A=\overline{s}_{N\widehat{1}}$,
 $-\overline{s}_{N\widehat{12}}\, \partial_2\bar{\partial}_2$,..., $-\overline{s}_{N\widehat{1n}}\,\partial_n\bar{\partial}_n$, we see that these operators are self-adjoint and positive definite in $H_{h\widehat{1}}$.
Therefore taking the imaginary part in (\ref{q1}), we obtain
\[
\frac{\hbar}{\tau}
\left(\overline{s}_{N\widehat{1}}Y^m,Y^m\right)_{\overline{\omega}_{hN1}
\times\omega_{h\widehat{1}}}
+
c_{\hbar}
\Ima \left(
{\mathcal S}^m_{\rm ref}{\mathbf Y}^m_{J_1},Y^m_{J_1}\right)_{\omega_{h\widehat{1}}}=0.
\]
Owing to Lemma \ref{ln} we get
\[
\frac{\hbar}{\tau}
\left(\overline{s}_{N\widehat{1}}Y^m,Y^m\right)_{\overline{\omega}_{hN1}\times\omega_{h\widehat{1}}}
\leq 0,
\]
and finally the above positive definiteness of \eqref{m2} and $\overline{s}_{N\widehat{1}}$ implies that $Y^m=0$.
\par Bound (\ref{o1}) follows directly from the previous $L^2$-stability bound (\ref{l2}) in the case of the infinite space mesh since now $\Psi^0_h=0$ and $F^m=0$ on $\omega_{h,\infty}\backslash \omega_h$ for any $m\geq 1$.
\end{proof}
\par For $F=0$, bound (\ref{o1}) means that
$\|\Psi^m\|_{\widetilde\omega_h}^2\leq  \|\Psi^0_{h}\|_{\widetilde\omega_h}^2$ for any $m\geq 1$.
\par Note that, in order to prove uniqueness of the solution, we have crucially exploited a very particular case of inequality (\ref{nn1}) (see also (\ref{nn2})), namely
\[
0\leq\Ima
\left(
{\mathcal F}_2^{-1}\dots{\mathcal F}_n^{-1}
\left[
\sigma^{(2)}_{q_2}\dots \sigma^{(n)}_{q_n}
R_{\mathbf q}^0 Y_{J_1}^{m{\mathbf q}}
\right],Y^m_{J_1}\right)_{\omega_{h\widehat{1}}}
\]
\[
=\frac{X_2\dots X_n}{2^{n-1}}
\sum_{q_2=1}^{J_2-1}\ldots\sum_{q_n=1}^{J_n-1}
\sigma^{(2)}_{q_2}\dots \sigma^{(n)}_{q_n}
\Ima R_{\mathbf q}^0 \left|Y^{m{\mathbf q}}_{J_1}\right|^2,
\]
for any $Y^m_{J_1}$: $\omega_{h\widehat{1}}\to{\mathbb C}$, which is equivalent to the inequality
$\Ima R_{\mathbf q}^0\geq 0$ for any $\mathbf q$.
\par The splitting scheme on the finite space mesh (\ref{l1})-(\ref{l6}) can be effectively implemented (similarly to \cite{DZZ13,ZR13}). Applying the operator ${\mathcal F}_2\ldots{\mathcal F}_n$ to the main equation (\ref{l2}) and the discrete TBC (\ref{l5}) and dividing the results by $\sigma^{(2)}_{q_2}\dots \sigma^{(n)}_{q_n}$, we get a collection of independent 1D problems in $x_1$, for each $\widetilde{\Psi}^{m{\mathbf q}}$
\begin{equation}
 i\hbar s_{N1}\,\frac{\widetilde{\Psi}^{m{\mathbf q}}-\breve{\Psi}^{m{\mathbf q}}}{\tau}
 =-c_{\hbar}\partial_1\bar{\partial}_1\frac{\widetilde{\Psi}^{m{\mathbf q}}+\breve{\Psi}^{m{\mathbf q}}}{2}
+s_{N1}\Bigl(\widetilde{V}_{\mathbf q}\frac{\widetilde{\Psi}^{m{\mathbf q}}+\breve{\Psi}^{m{\mathbf q}}}{2}\Bigr)
+\frac{F^{m{\mathbf q}}}{\sigma^{(2)}_{q_2}\dots \sigma^{(n)}_{q_n}}\ \ \mbox{on}\ \ \omega_{h1},
\label{r1}
\end{equation}
\begin{equation}
\left.   \widetilde{\Psi}^{m{\mathbf q}}\right|_{j_1=0}=0,
\label{r2}
\end{equation}
\begin{equation}
\Bigl.
\Bigl[
c_{\hbar}\bar{\partial}_1\,\frac{\widetilde{\Psi}^{m{\mathbf q}}
+\breve{\Psi}^{m{\mathbf q}}}{2}
-h_1s_{N1}^-
\Bigl(
i\hbar\,\frac{\widetilde{\Psi}^{m{\mathbf q}}-\breve{\Psi}^{m{\mathbf q}}}{\tau}
-V_{\infty,{\mathbf q}}\,\frac{\breve{\Psi}^{m{\mathbf q}}
+\widetilde{\Psi}^{m{\mathbf q}}}{2}
\Bigr)
\Bigr]
\Bigr|_{j_1=J_1}
=c_{\hbar}
\Bigl(
R_{\mathbf q} *\widetilde{\mathbf \Psi}^{\mathbf q}_{J_1}
\Bigr)^m,
\label{r3}
\end{equation}
where $\widetilde{V}_{\mathbf q}:=\widetilde{V}+c_{\hbar}\Bigl(
\frac{\lambda^{(2)}_{q_2}}{\sigma^{(2)}_{q_2}}+\dots +\frac{\lambda^{(n)}_{q_n}}{\sigma^{(n)}_{q_n}}
\Bigr)$ and we have taken into account (\ref{l1}) and (\ref{l7}).
\par Given $\Psi^{m-1}$, the direct algorithm for computing $\Psi^m$ is divided into five steps.
\begin{enumerate}
\item To compute $\breve{\Psi}^m={\mathcal E}^m\Psi^{m-1}$ on $\omega_h\cup\Gamma_{1h}$ (see (\ref{e6})).
\item To compute
$\breve{\Psi}^{m{\mathbf q}}
=\bigl(
{\mathcal F}_n\dots
\bigl(
{\mathcal F}_2 \breve{\Psi}^m
\bigr)^{(q_2)}
\dots
\bigr)^{(q_n)}$ and
$F^{m{\mathbf q}}=\bigl({\mathcal F}_n\dots\bigl({\mathcal F}_2 F^m\bigr)^{(q_2)}\dots\bigr)^{(q_n)}$,
for $1\leq q_2\leq J_2-1,\dots,1\leq q_n\leq J_n-1$.
\item  To compute
$\widetilde{\Psi}^{m{\mathbf q}}$ by solving the independent 1D problems (\ref{r1})-(\ref{r3})
for $1\leq q_2\leq J_2-1,\dots,1\leq q_n\leq J_n-1$ (this includes the computation of the discrete convolutions on the right
of (\ref{r3}) so that
$\widetilde{\Psi}_{J_1}^{1{\mathbf q}},\dots,\widetilde{\Psi}_{J_1}^{m-1\, {\mathbf q}}$ have to be stored).
\item  To compute
$\widetilde{\Psi}^m={\mathcal F}_n^{-1}\dots {\mathcal F}_2^{-1} \widetilde{\Psi}^{m{\mathbf q}}$.
\item To compute
$\Psi^m={\mathcal E}^m\widetilde{\Psi}^m$ on $\omega_h\cup\Gamma_{1h}$ (see (\ref{e6})).
\end{enumerate}

Steps 1 and 5 need $O\left(J_1\dots J_n\right)$ arithmetic operations while
Steps 2 and 4 require $O\left(J_1\dots J_n\log_2\left(J_2\dots J_n\right)\right)$ operations by using FFT provided that
$J_2=2^{k_2},\dots,J_n=2^{k_n}$, where $k_2,\dots,k_n$ are integers.
Step 3 needs $O((J_1+m)J_2$ $\dots J_n)$ operations.
\par The total amount of arithmetic operations equals
$O\left((J_1\log_2\left(J_2\dots J_n)+m\right)J_2\dots J_n\right)$ or
$O(\left(J_1\log_2\left(J_2\dots J_n\right)+M\right)J_2\dots J_n M)$
in order to compute the solution $\Psi^m$ respectively at time level $m$ or
at all time levels $m=1,\dots,M$.
\par Notice that the above analysis is easily extended to the case of the problem in a parallelepiped infinite in $x_1$ in both directions, with setting the discrete TBC at the left artificial boundary $x_1=0$ as well. Its form similar to (\ref{r3}) is as follows
\begin{equation}
\Bigl.
\Bigl[
-c_{\hbar}{\partial}_1\,\frac{\widetilde{\Psi}^{m{\mathbf q}}
+\breve{\Psi}^{m{\mathbf q}}}{2}
-h_1s_{N1}^+
\Bigl(
i\hbar\,\frac{\widetilde{\Psi}^{m{\mathbf q}}-\breve{\Psi}^{m{\mathbf q}}}{\tau}
-V_{\infty,{\mathbf q}}\,\frac{\breve{\Psi}^{m{\mathbf q}}
+\widetilde{\Psi}^{m{\mathbf q}}}{2}
\Bigr)
\Bigr]
\Bigr|_{j_1=0}
=c_{\hbar}
\Bigl(
R_{\mathbf q} *\widetilde{\mathbf \Psi}^{\mathbf q}_{0}
\Bigr)^m,
\label{r3m}
\end{equation}
for any $m\geq 1$ and ${\mathbf q}$ (for brevity, we suppose that $V(x)=V_\infty$ also for $x_1\leq h_1$ though clearly $V_{\pm\infty}$ could be different). Here
$\widetilde{\mathbf \Psi}_{0}=\{\widetilde{\Psi}^0|_{j_1=0},\dots,\widetilde{\Psi}^m|_{j_1=0}\}$.
\section{Numerical experiments}
\setcounter{equation}{0}
\setcounter{proposition}{0}
\setcounter{theorem}{0}
\setcounter{lemma}{0}
\setcounter{corollary}{0}
\setcounter{remark}{0}
The above presented direct algorithm has been implemented for $n=2$.
We solve the initial-boundary value problem in the infinite strip $\mathbb{R}\times (0,X_2)$ taking the computational domain $\Pi_X\times [0,T]$, with $\Pi_X=[0,X_1]\times [0,X_2]$, and set $\hbar=1$ and $c_\hbar=1$.
\par We respectively modify our scheme  (\ref{l1})-(\ref{l6}) enlarging
$\omega_h\cup \Gamma_{1h}$ by $\Gamma_{0h}:=\{0\}\times\omega_{h\widehat{1}}$ in \eqref{l1} and \eqref{l3} as well as
replacing $\Gamma_h$ by $\Gamma_h\setminus\Gamma_{0h}$ in \eqref{l4}
and posing the left discrete TBC \eqref{r3m}.
We can put $\widetilde{V}=0$ and $\Delta V=V$.\par Let the initial function be the standard Gaussian wave package
\[
 \psi^0(x)=
 \psi_G:=
 \exp\Bigl\{ ik(x_1-x_1^{(0)})-\frac{|x-x^{(0)}|^2}
 {4\alpha}\Bigr\}\ \ \text{on}\ \ \mathbb{R}^2.
\]
We set its parameters $k=30\sqrt{2}$ (the wave number), $\alpha=\frac{1}{120}$ and $x^{(0)}=(1,\frac{X_2}{2})$
like in \cite{DZZ13,ZR13}.
The modulus and the real part of $\psi_G$ can be seen on Figure \ref{B:Solution1}, for $m=0$.
\smallskip\par\textbf{Example A.}
We first consider a modified P\"{o}schl-Teller \cite{F71} potential-barrier
\[
 V(x)=\frac{\alpha_0^2 c_1}{\cosh^2 \alpha_0(x_1-x_1^*)}
\]
depending only on $x_1$ and set $\alpha_0=6$, $c_1=47$ and $x_1^*=2$.
Though the potential is smooth, its derivatives in $x_1$ are rather large.
Let $(X_1,X_2)=(4,4.2)$, then both $V$ and $\psi_G$ are sufficiently small outside $\Pi_X$, and let also $T=t_M=0.05$.
\par This example was solved using the Numerov-Crank-Nicolson scheme with the same Strang splitting in potential in \cite{ZR13} on various meshes. The wave package is divided by the barrier into two comparable reflected and transmitted parts moving in opposite $x_1$-directions and leaving the computational domain.
In particular, it was found that values $(J_1,J_2,M)=(400,64,1000)$, i.e.
$h_1=10^{-2}$, $h_2\approx6.56\cdot10^{-2}$ and $\tau=\frac{T}{M}=5\cdot 10^{-5}$, are suitable to build correct graphs of the solution;
finer meshes allowed to compute much more precise numerical solutions.
\par Here we study the difference between the numerical solutions of the Numerov-Crank-Nicolson-Strang scheme and the above one.
In Tables \ref{Table1} and \ref{Table2}, we present their maximum in time $C$ and $L^2$ space norms $E_C$ and $E_{L^2}$ on refining space meshes together with the corresponding ratios $R_C$ and $R_{L^2}$.
Notice that the difference is estimated theoretically as $O(h_1^2h_2^2)$. So it is natural that $R_C$ and $R_{L^2}$ are rather close to 4 for redoubling $J_1$ or $J_2$ and to 16 for redoubling $(J_1,J_2)$.
\begin{table}
\begin{center}
\begin{tabular}{|r|c|c|c|c|c|}
  \hline
  $J_1$ & $E_C$  & $E_{L^2}$  & $R_C$     & $R_{L^2}$ \\
  \hline
  200  & $0.340$$E$$-2$  &$0.330$$E$$-2$    & --        & --      \\
  400  & $0.851$$E$$-3$  &$0.810$$E$$-3$    & 3.99    & 4.07  \\
  800  & $0.213$$E$$-3$  &$0.202$$E$$-3$    & 3.99    & 4.01  \\
  1600 & $0.531$$E$$-4$  &$0.505$$E$$-4$    & 4.01    & 4.00  \\
  \hline
\end{tabular}
\end{center}
\begin{center}
\begin{tabular}{|r|c|c|c|c|c|}
  \hline
  $J_2$ & $E_C$  & $E_{L^2}$  & $R_C$     & $R_{L^2}$ \\
  \hline
  32  & $0.410$$E$$-2$  &$0.310$$E$$-2$    & --      & --                \\
  64  & $0.853$$E$$-3$  &$0.807$$E$$-3$    & 4.81    & 3.84            \\
  128 & $0.213$$E$$-3$ & $0.202$$E$$-3$    & 4.00    & 3.99            \\
  256 & $0.531$$E$$-4$ & $0.506$$E$$-4$    & 4.01    & 3.99  \\
  \hline
\end{tabular}
\end{center}
\caption{The difference between the solutions of two schemes in maximum in time $C$ and $L^2$ space norms for redoubling
$J_1$ and $(J_2,M)= (128,1000)$,
or redoubling $J_2$ and $(J_1,M)= (800,1000)$.}
\label{Table1}
\end{table}
\begin{table}
\begin{center}
\begin{tabular}{|c|c|c|c|c|c|}
  \hline
  $(J_1,J_2)$ & $E_C$  & $E_{L^2}$  & $R_C$     & $R_{L^2}$ \\
  \hline
  $(400,64)$   & $0.340$$E$$-2$ & $0.320$$E$$-2$    & --      & --          \\
  $(800,128)$  & $0.213$$E$$-3$ & $0.202$$E$$-3$    & 15.96   & 15.84       \\
  $(1600,256)$ & $0.193$$E$$-4$ & $0.129$$E$$-4$    & 11.03   & 15.66       \\
  \hline
\end{tabular}
\end{center}
\caption{Example A. The difference between the solutions of two schemes in maximum in time $C$ and $L^2$ space norms for redoubling $(J_1,J_2)$ and $M=1000$.}
\label{Table2}
\end{table}
\par The typical graphs in time of the absolute and relative differences in $C$ and $L^2$ norms between the numerical solutions of two schemes are given on Figure \ref{A:Errors} for $(J,K,M)=(800,128,1000)$.
\begin{figure}[ht]
\begin{multicols}{2}
    \includegraphics[scale=0.4]{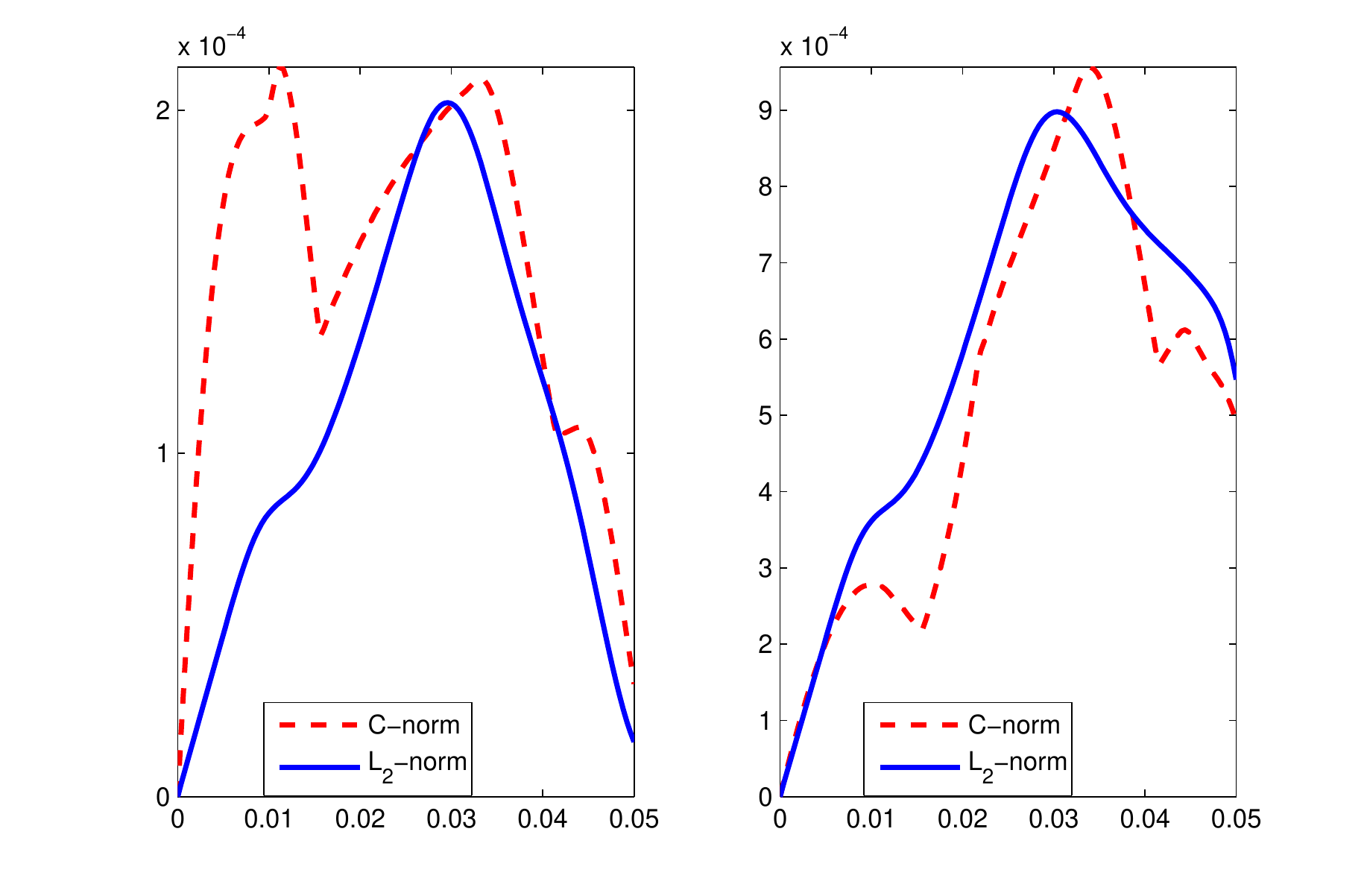}\vspace{0cm}\\
    \caption{\small{Example A.
The absolute (left) and relative (right) differences in $C$ and $L^2$ norms between the solutions of two schemes for
$(J_1,J_2,M)=(800,128,1000)$ in dependence with time}}
\label{A:Errors}
    \includegraphics[scale=0.4]{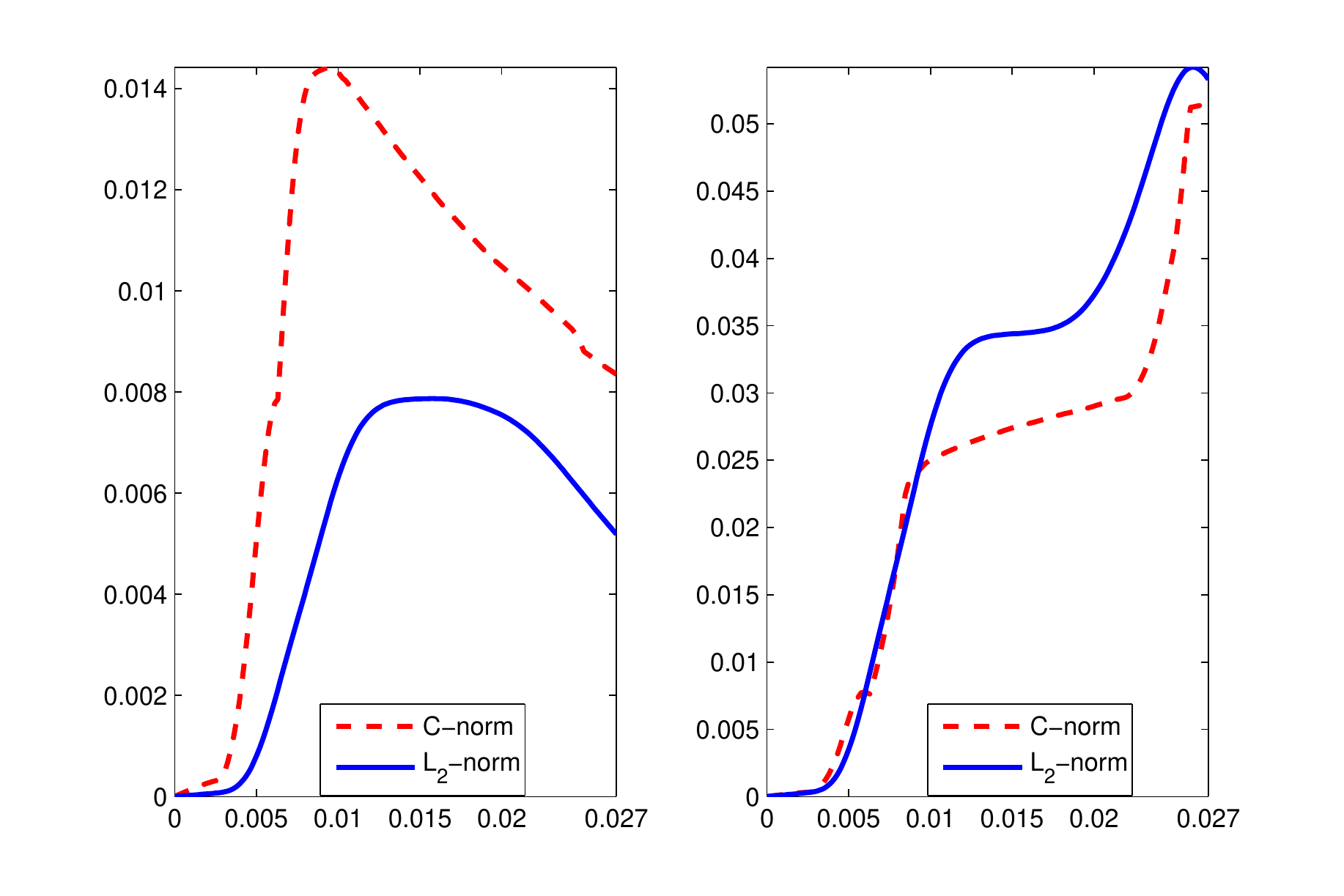}\vspace{0cm}\\
\caption{\small{Example B.
The absolute (left) and relative (right) differences in $C$ and $L^2$ norms between the numerical solutions for
$(J_1,J_2,M)=(600,64,2400)$ and $(1200,128,4800)$ in dependence with time}}
\label{B:Errors}
\end{multicols}
\end{figure}
\smallskip\par\textbf{Example B.}
Following \cite{DZZ13,ZR13}, we second consider the rectangular potential
\[
 V(x)=
 \begin{cases}
 Q &\text{for}\ \ x\in \Pi:=(a,b)\times (c,d)\\[1mm]
 0 &\text{otherwise}
\end{cases}
\]
depending both on $x_1$ and $x_2$.
We set $\Pi=(1.6,1.9)\times (0.7,2.1)$ and $Q=-9000$ so now the potential is a well (in contrast to \cite{DZZ13,ZR13}).
We choose $(X_1,X_2)=(3,2.8)$ so that $\Pi\subset \Pi_X$ and $\psi_G$ is small outside $\Pi_X$. Let also $T=t_M=0.027$. This example is more complicated since the well is discontinuous and thus the corresponding exact solution is non-smooth.
\par We take $(J_1,J_2)$ such that the vertices of $\Pi$ belong to the mesh and, following \cite{ZR13}, exploit the averaged mesh potential
\begin{gather*}
 V_{hj_1,j_2}=\begin{cases}
 V(x_{{\bf j}})   &\text{for}\ \ j_1h_1\neq a,b\ \ \text{and}\ \ j_2h_2\neq c,d\\[1mm]
 Q/2 &\text{for}\ \ j_1h_1=a,b\ \ \text{but}\ \ j_2h_2\neq c,d,
     \ \text{or for}\ \ j_2h_2=c,d\ \ \text{but}\ \ j_1h_1\neq a,b\\[1mm]
 Q/4 &\text{for}\ \ (j_1h_1,j_2h_2)=(a,c),(a,d),(b,c),(b,d)
\end{cases}
\end{gather*}
for any $j_1$ and $j_2$.
\par The numerical solution $\Psi^{m}$ is computed for $(J_1,J_2,M)=(600,64,2400)$, i.e.,
$h_1=5\cdot10^{-3}$, $h_2=4.375\cdot10^{-2}$ and $\tau=1.125\cdot 10^{-5}$.
We check that these values are suitable by computing the change in the solution when redoubling $(J_1,J_2,M)$, see Figure \ref{B:Errors}.
The modulus and the real part of $\Psi^{m}$ together with the normalized well are presented on Figures \ref{B:Solution1} and \ref{B:Solution2}, for some selected time levels.
Once again the wave package is divided (now by the well) into the reflected and transmitted parts,
but now the process is more tricky and the reflected part consists in two fragments. Notice (as usual) the more complicated behavior of the real part and the complete absence of the spurious reflections from the artificial left and right boundaries where the discrete TBCs are posed.
\begin{figure}[ht]
\begin{multicols}{2}
    \includegraphics[scale=0.5]{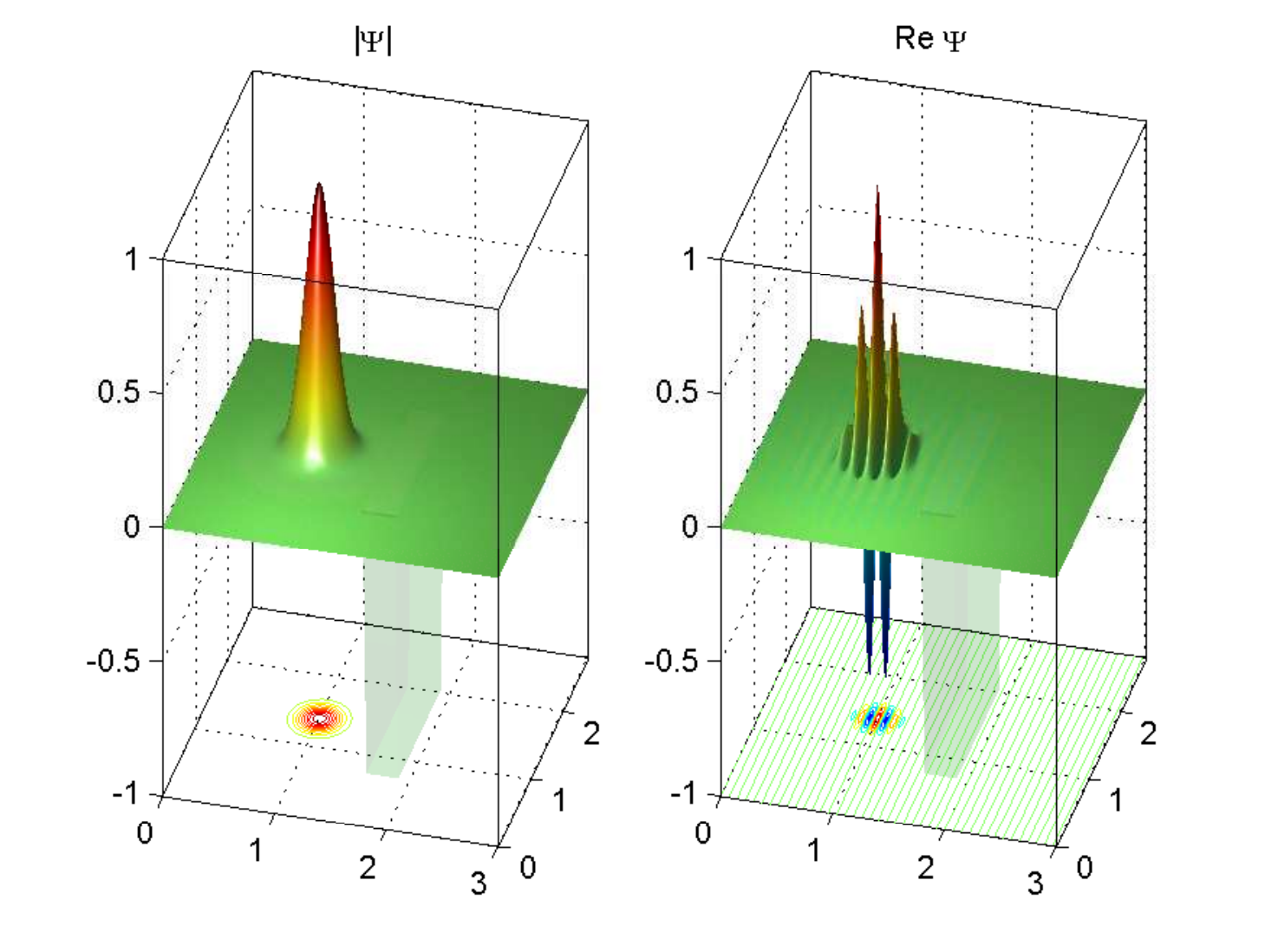}\vspace{0cm}\\
    \centerline{\small{$m=0$}}\\
    \includegraphics[scale=0.5]{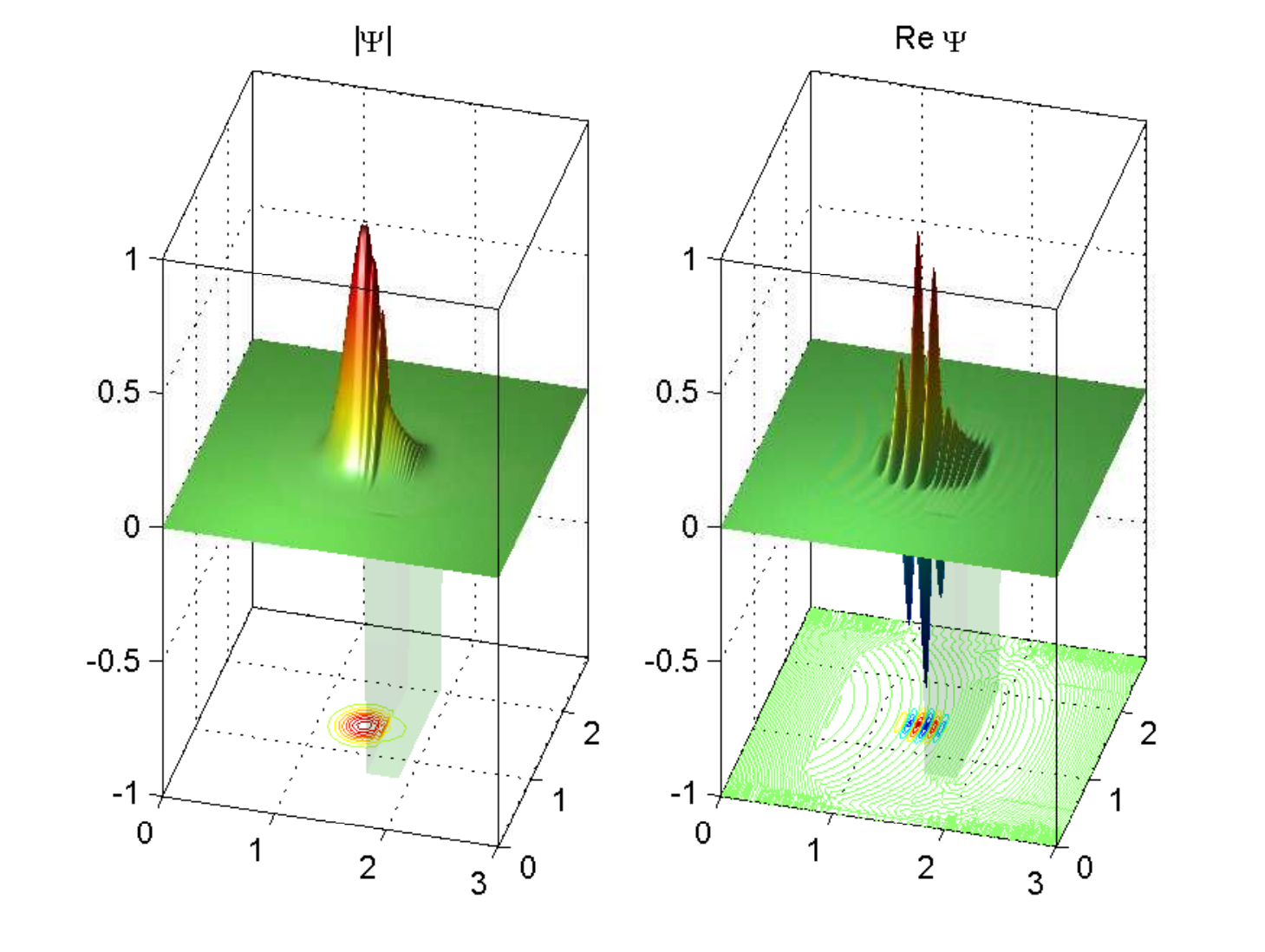}\vspace{0cm}\\
    \centerline{\small{$m=416$}}\\
   \end{multicols}
\begin{multicols}{2}
    \includegraphics[scale=0.5]{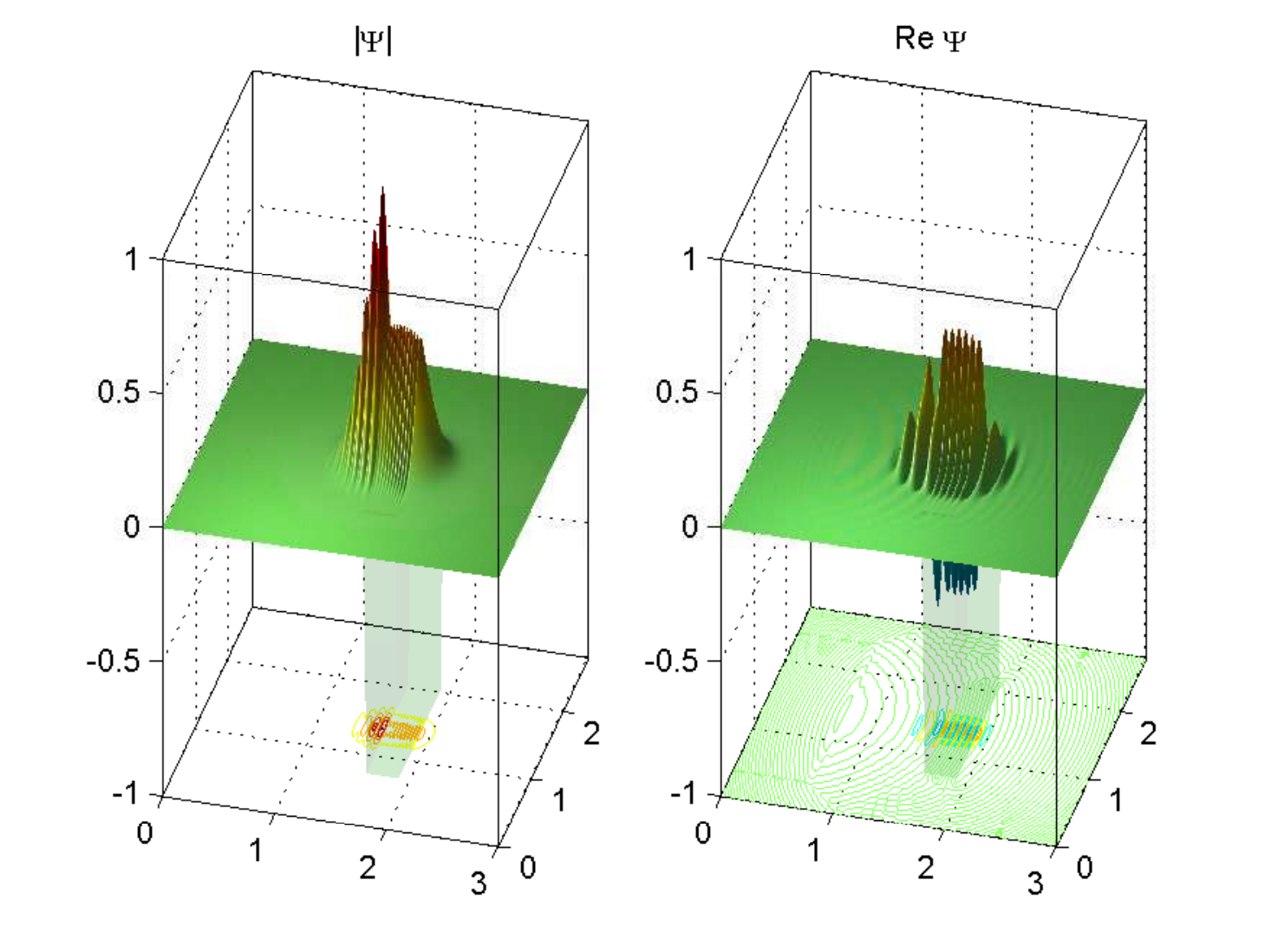}\vspace{0cm}\\%
    \centerline{\small{$m=616$}}\\
    \includegraphics[scale=0.5]{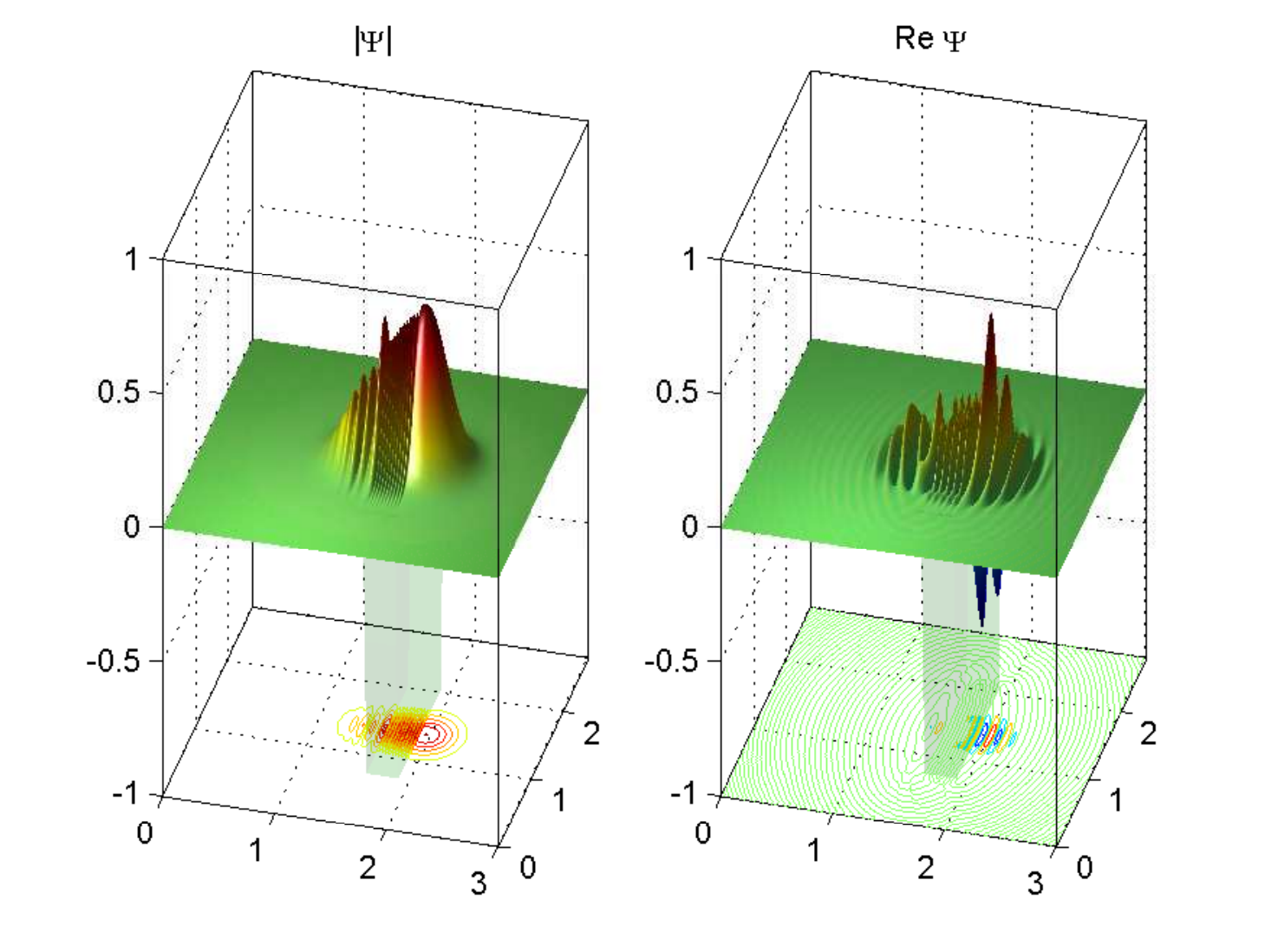}\vspace{0cm}\\%
     \centerline{\small{$m=818$}}\\
\end{multicols}
\caption{\small{Example B. The modulus and the real part of the numerical solution $\Psi^m$, $m=0, 416, 616$ and $818$}}
\label{B:Solution1}
\end{figure}

\begin{figure}[ht]
\begin{multicols}{2}
\includegraphics[scale=0.5]{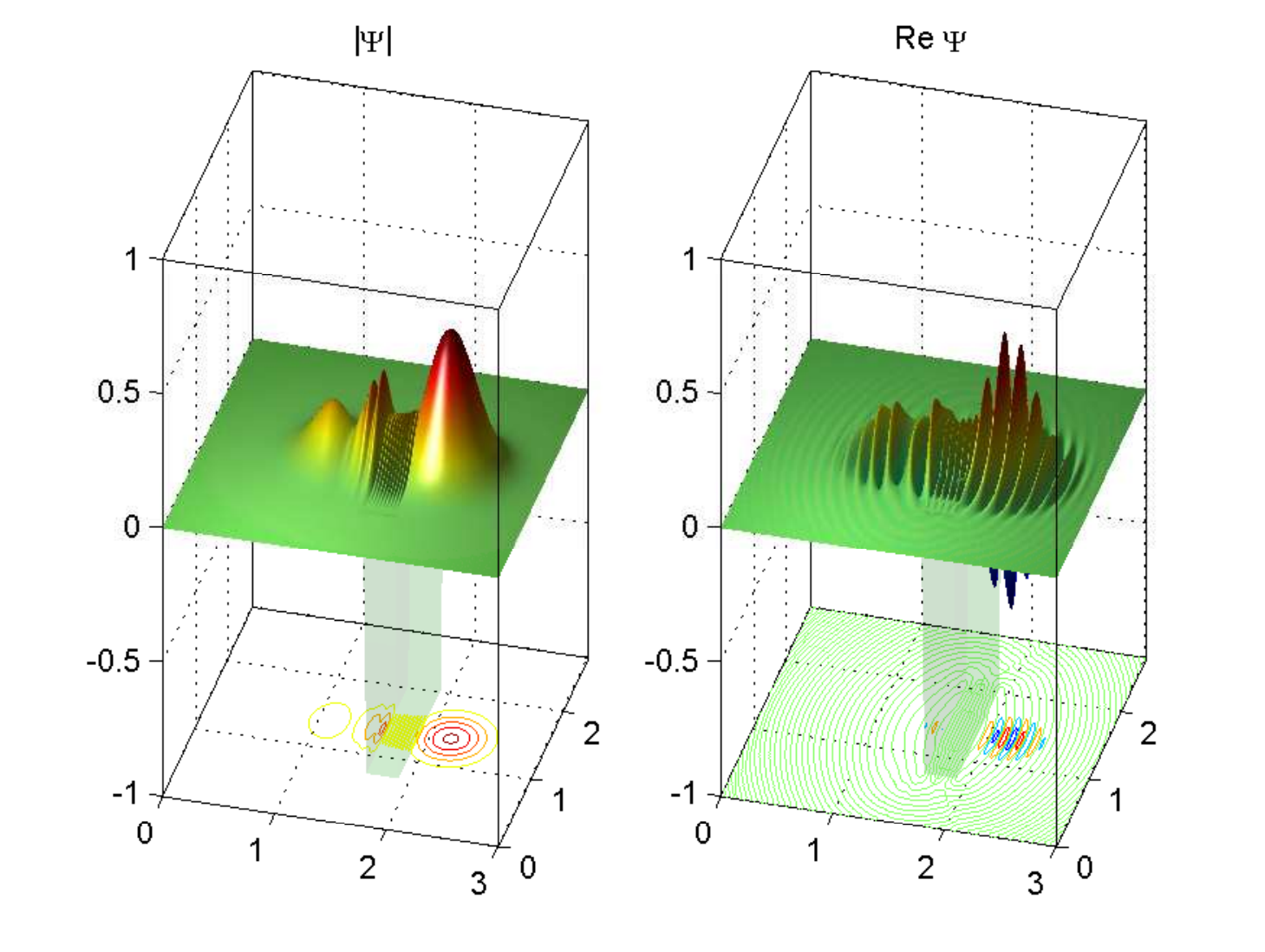}\vspace{0cm}\\%
     \centerline{\small{$m=1056$}}\\
    \includegraphics[scale=0.5]{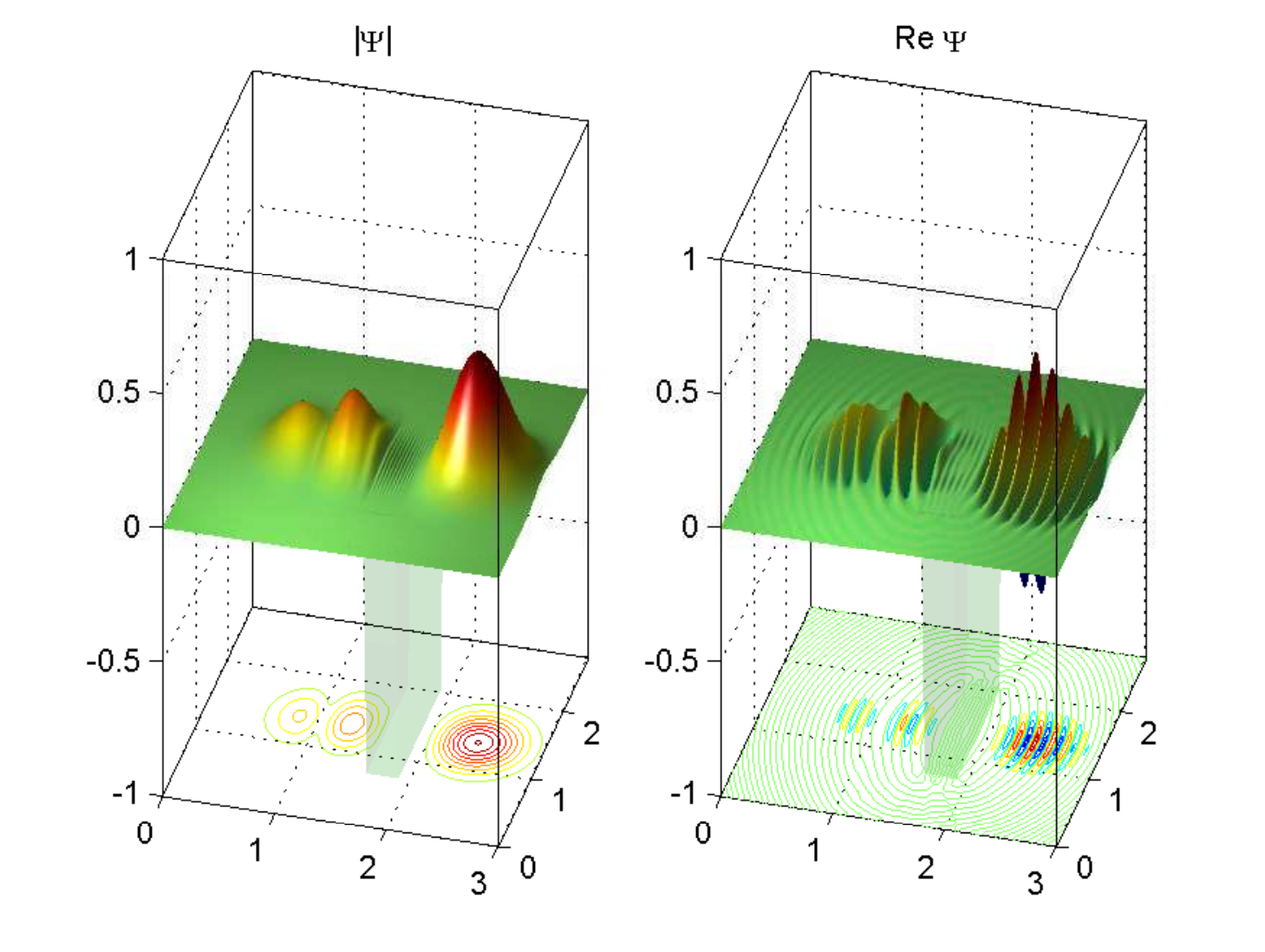}\vspace{0cm}\\%
     \centerline{\small{$m=1318$}}\\
\end{multicols}

\begin{multicols}{2}
    \includegraphics[scale=0.5]{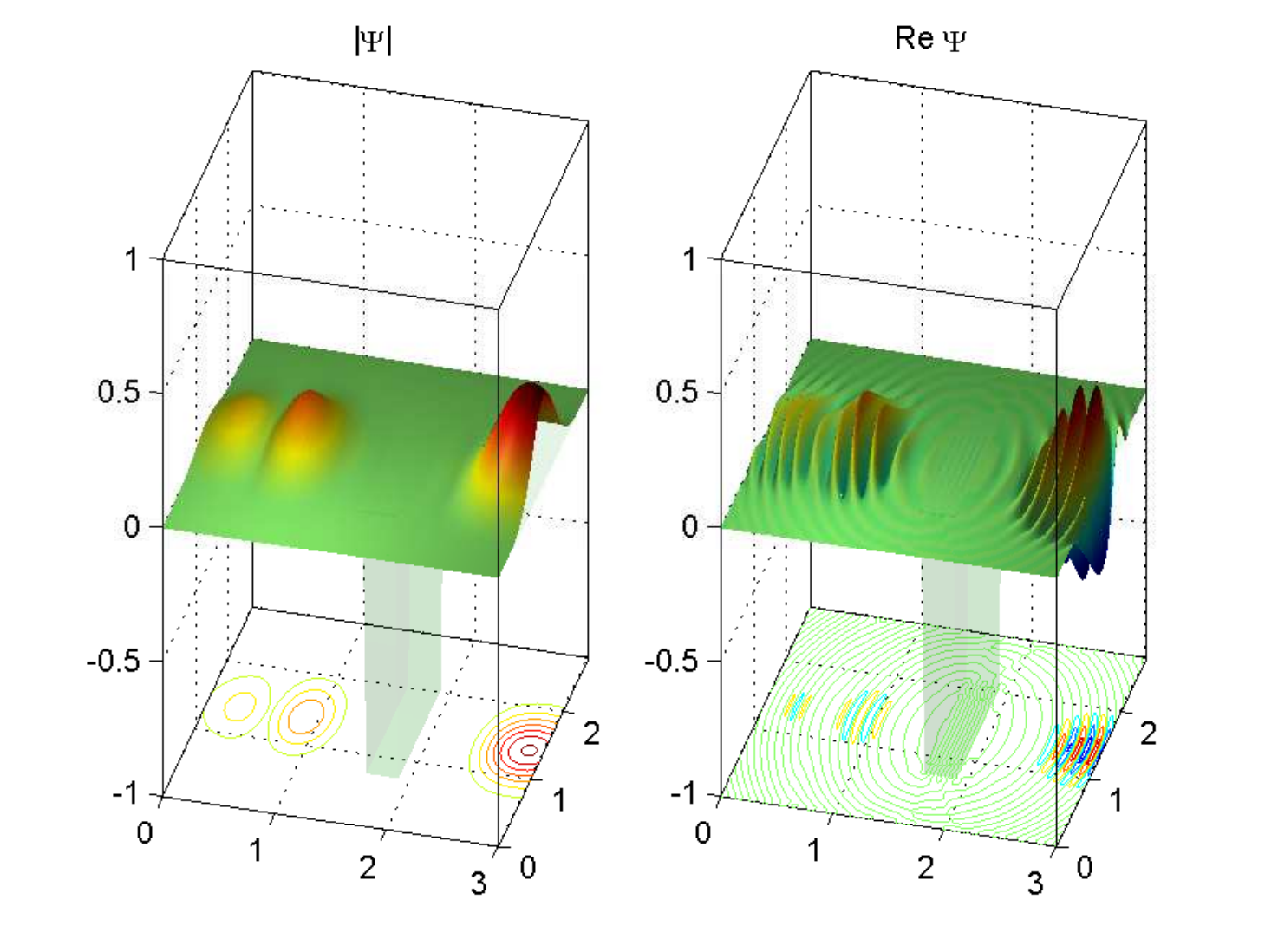}\vspace{0cm}\\%
     \centerline{\small{$m=1800$}}\\
    \includegraphics[scale=0.5]{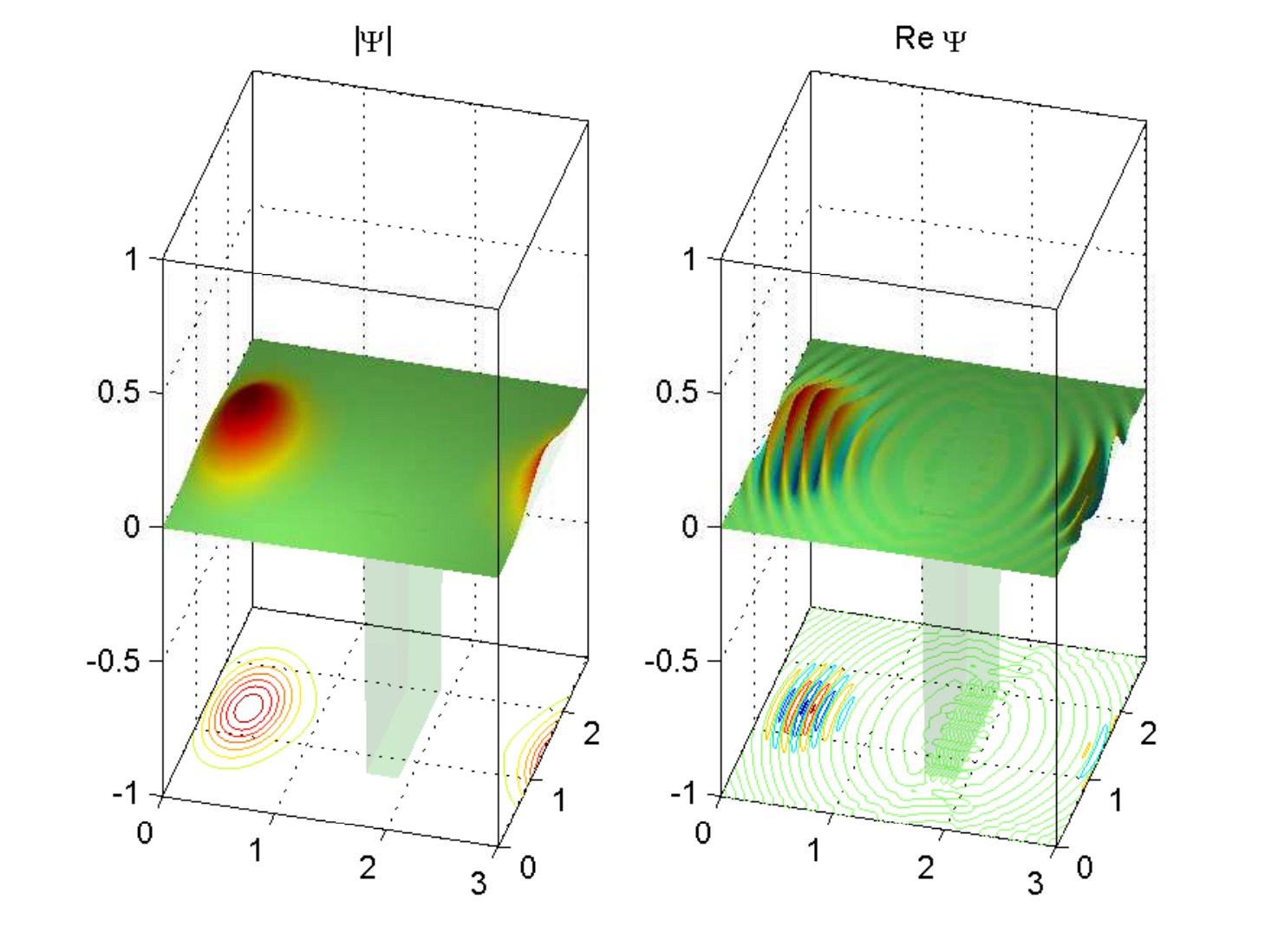}\vspace{0cm}\\%
     \centerline{\small{$m=2400$}}\\
\end{multicols}
\caption{\small{Example B. The modulus and the real part of the numerical solution $\Psi^m$, $m=1056, 1318, 1800$ and $2400$}
\label{B:Solution2}}
\end{figure}
\par On the last Figure \ref{energies}, the graphs of the total kinetic and potential energies are presented. Here we calculate them as
\[
 E_{kin}:=c_\hbar\Bigl(\|\bar{\partial}_1\Psi\|_{\widetilde{\omega}_h}^2
 +\sum_{j_1=1}^{J_1}\sum_{j_2=1}^{J_2}|\bar{\partial}_2\Psi_{j_1,j_2}|^2h_1h_2\Bigr)\ \ \text{for}\ \ n=2,\ \
 E_{pot}:=(V\Psi,\Psi)_{\widetilde{\omega}_h}.
\]
The left and the right graphs correspond respectively to Example B and the related example from \cite{DZZ13,ZR13} for the rectangular \textit{barrier} with $\Pi=(1.6,1.7)\times (0.7,2.1)$ and $Q=1500$.
Their behavior is in complete accordance with the physical sense of the examples.
\begin{figure}[ht]
\begin{multicols}{2}
    \includegraphics[scale=0.4]{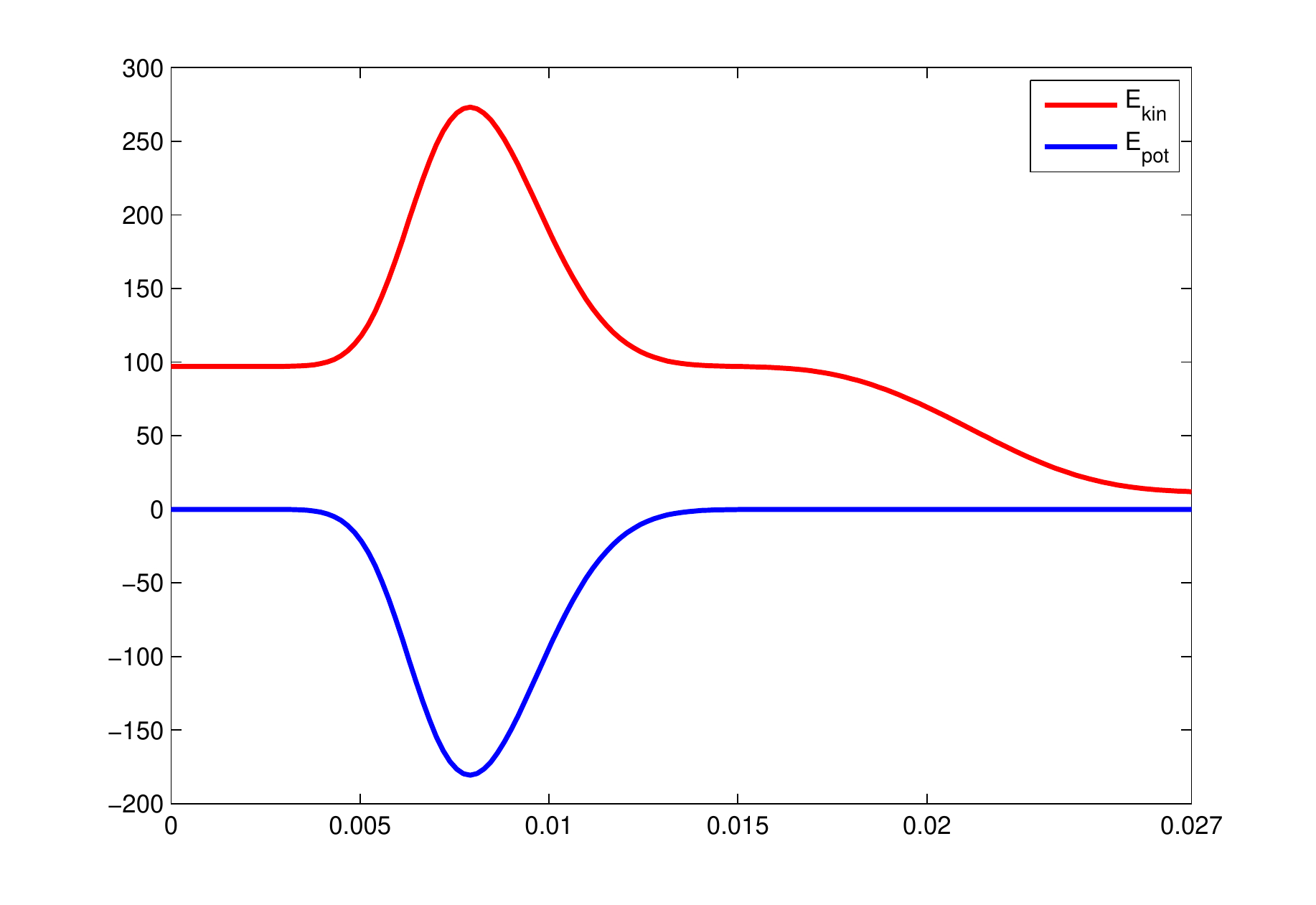}\vspace{0cm}\\%
    \includegraphics[scale=0.4]{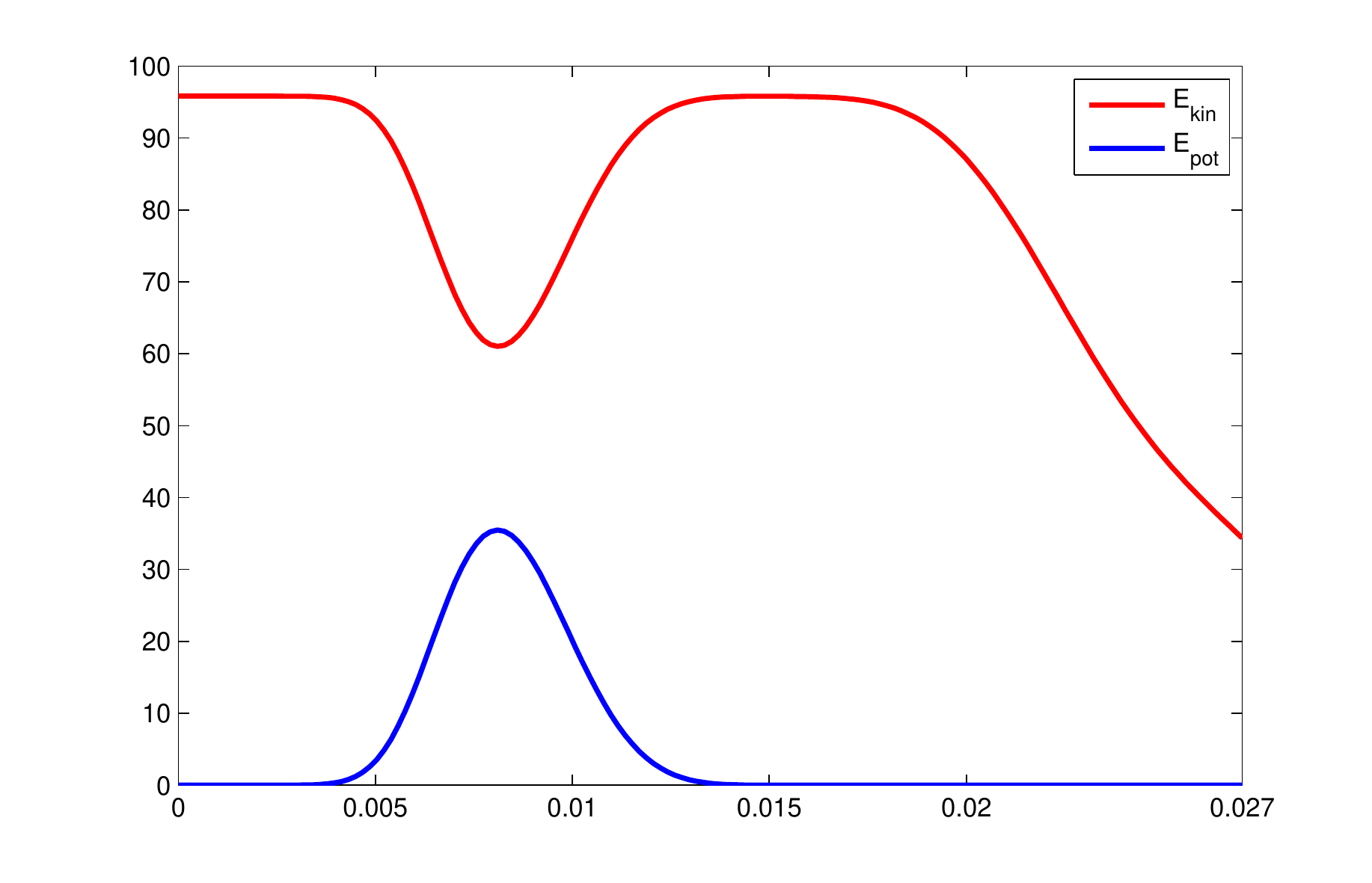}\vspace{0cm}\\%
\end{multicols}
\caption{\small{Example B. The total kinetic and potential energies in the cases of the well (left) and the barrier (right) in time}}
\label{energies}
\end{figure}
\bigskip\par
\textbf{Acknowledgments}
\smallskip\par
The paper has been initiated during the visit of A. Zlotnik in summer 2012 to the the D\'epar\-te\-ment de Physique
Th\'eorique et Appliqu\'ee, CEA/DAM/DIF Ile de France (Arpajon), which he thanks for hospitality.
The study is carried out by A. Zlotnik and A. Romanova
within The National Research University Higher School of Economics' Academic Fund Program, project No. 13-09-0124
and is also supported by the Russian Foundation for Basic Research, project  No. 12-01-90008-Bel.

\end{document}